%% file: geclassify2.tex
\documentclass[10pt]{amsart}

\usepackage{enumerate}
\usepackage{color}
\usepackage{graphics}
\usepackage{amssymb}
\usepackage{amscd}
\usepackage{xypic}

\newtheorem{proposition}{Proposition}[section]
\newtheorem{lemma}[proposition]{Lemma}
\newtheorem{corollary}[proposition]{Corollary}
\newtheorem{theorem}[proposition]{Theorem}
\newtheorem{remark}[proposition]{Remark}

\newtheorem{definition}[proposition]{Definition}

\newtheorem{example}[proposition]{Example}

\newcommand{\thlabel}[1]{\label{th:#1}}

\newcommand{\selabel}[1]{\label{se:#1}}
\newcommand{\seref}[1]{Section~\ref{se:#1}}
\newcommand{\lelabel}[1]{\label{le:#1}}
\newcommand{\leref}[1]{Lemma~\ref{le:#1}}
\newcommand{\prlabel}[1]{\label{pr:#1}}
\newcommand{\prref}[1]{Proposition~\ref{pr:#1}}
\newcommand{\colabel}[1]{\label{co:#1}}
\newcommand{\coref}[1]{Corollary~\ref{co:#1}}
\newcommand{\relabel}[1]{\label{re:#1}}
\newcommand{\reref}[1]{Remark~\ref{re:#1}}
\newcommand{\exlabel}[1]{\label{ex:#1}}

\newcommand{\delabel}[1]{\label{de:#1}}

\newcommand{\eqnlabel}[1]{\label{eqn:#1}}

\newcommand{\id}{\mathop{\mathrm{id}}\nolimits}
\newcommand{\tr}{\mathop{\mathrm{tr}}\nolimits}
\newcommand{\sgn}{\mathop{\mathrm{sgn}}\nolimits}
\newcommand{\ch}{\mathop{\mathrm{ch}}\nolimits}
\newcommand{\spn}{\mathop{\mathrm{span}}\nolimits}
\newcommand{\Salg}{\mathop{\mathrm{Salg}}\nolimits}

\newcommand{\SA}{\mathop{\mathrm{SA}}\nolimits}
\newcommand{\Alg}{\mathop{\mathrm{Alg}}\nolimits}

\newcommand{\GL}{\mathop{\mathrm{GL}}\nolimits}

\newcommand{\Stab}{\mathop{\mathrm{Stab}}\nolimits}

\newcommand{\Aut}{\mathop{\mathrm{Aut}}\nolimits}
\newcommand{\qmra}{\stackrel{?} \rightarrow}
\def\A{\mathbb{A}}
\def\a{\alpha}
\def\b{\beta}
\def\g{\gamma}

\def\d{\delta}

\def\i{\iota}

\def\l{\lambda}
\def\L{\Lambda}

\def\Om{\Omega}

\def\s{\sigma}

\def\p{\phi}

\def\bs{\backslash}
\def\ot{\otimes}

\begin{document}
\title{Geometric Classification of 4-Dimensional Superalgebras}
\author{Aaron Armour}
\address{School of Mathematics, Statistics and Computer Science,
Victoria  University of Wellington, PO Box 600,
Wellington, New Zealand}
\email{aaron.armour@mcs.vuw.ac.nz}
\author{Yinhuo Zhang}
\address{Department of Mathematics and Statistics,
University of Hasselt, B-3590 Diepenbeek, Belgium}
\email{yinhuo.zhang@uhasselt.be}
\subjclass{16W55,16W30, 16W50}
\keywords{Superalgebra, degeneration, irreducible component}

\maketitle

\section{Introduction}
The algebraic and geometric classification of finite dimensional algebras over an algebraic closed field $k$ was initiated  by Gabriel in \cite{Gab}, and has been being one of the interesting topics in the study of geometric methods in representation theory of algebras for the last three decades.  In \cite{Gab}, Gabriel gave a complete list of nonisomorphic 4-dimensional algebras over an algebraic closed field $k$ with characteristic not equal to $2$. The number of irreducible components of the variety Alg$_4$ is 5. The
classification of 5-dimensional $k$-algebras was done by Mazzola
in \cite{Maz}. The number of irreducible components of the variety Alg$_5$ was showed to be 10. Further studies on the classification of low
dimensional (rigid) algebras can be found in \cite{Flan, GM, Mak,Mak2, Rie}.
With the dimension $n$ increasing, both algebraic and geometric classifications of $n$-dimensional $k$-algebras become more and more difficult.
However, A lower and a upper bound for the number of irreducible
components of Alg$_n$ can be given (see \cite{Maz3}). Now let $V_n$ be an $n$-dimensional vector
space over $k$ with a basis $\{e_1,e_2, \cdots, e_n\}$. An
algebra structure on $V_n$ is determined by a set of
structure constants $c^h_{ij}$, where $e_i\cdot e_j=\sum_h
c^h_{ij}e_h$. Requiring the algebra structure to be associative and
unitary gives rise to a subvariety Alg$_n$ of
$k^{n^3}$. Base changes in $V_n$ result in the natural transport
of structure action on Alg$_n$, namely the action of GL$_n(k)$ on
Alg$_n$. Thus isomorphism classes of $n$-dimensional algebras are
in one-to-one correspondence with the orbits of the action of
$\GL_n(k)$ on Alg$_n$. The decomposition of Alg$_n$ into its
irreducible components under the Zariski topology is called the
geometric classification of $n$-dimensional algebras.

Our main interest is to give a geometric classification of
4-dimensional superalgebras, i.e. $\mathbb{Z}_2$-graded algebras. We notice that a $\mathbb{Z}_2$-graded  algebra is the same as a pair $(A, \sigma)$ consisting of an algebra $A$ and an algebra  involution $\sigma$. This enables us to define the variety $\Salg_n$ --- the variety of $n$-dimensional superalgebras --- as a subvariety of $k^{n^3+n^2}$. One of the significant differences between the variety Alg$_n$ and the variety Salg$_n$ is that Salg$_n$ is disconnected while Alg$_n$ is connected. Under certain assumptions on $n$ and $\ch(k)$ it can be shown that $\Salg_n$ is the disjoint union of $n$ connected
subvarieties, for example, when $n\leq 6$ or $ch(k)=0$.

The paper is organized as follows.  In Section 2, we define the variety Salg$_n$ of $n$-dimensional superalgebras, a closed subvariety of SA$_n$ of $n$-dimensional algebras. Salg$_n$ is a disjoint union of the subsets Salg$_n^i$, $i=1,2\cdots, n$. When $n\leq 6$ or $ch(k)=0$, they are closed in Salg$_n$ and form connected components of Salg$_n$.

In Section 3, we compute the dimensions of the orbits in Salg$_4$, which will help us to determine the degenerations of the superalgebras. So we need to recall the algebraic group $G_n$ and the transport of its structure action on Salg$_n$. Since $G_n$ is connected, every irreducible component of Salg$_n$ is the closure of either a single orbit or an infinite family of orbits.
In Section 4, we will use the ring properties of superalgebras to determine some closed sets of Salg$_n$. For instance, the set of superalgebras $A$ with $A_1^2=0$ is a closed subset. Similarly the set of superalgebras with $A_0$ being commutative is also a closed subset. The closed subsets can help us to determine some superalgebras that can not degenerate to other superalgebras.

In the last section, we give the degeneration diagrams of Salg$^i_4$, where $i=2,3$. The degeneration diagram of Salg$^4_4=\mathrm{Alg}_4$ has been given by Gabriel, and Salg$^1_4$ has only one orbit. In total, we  have found 20 irreducible components of Salg$_4$. However, Salg$_4$ may posses up to 22 irreducible components.

To end the introduction, let us recall from \cite{ACZ} the algebraic classification of 4-dimensional superalgebras over an algebraic closed field $k$  as the geometric classification must be made on the basis of the algebraic classification. Throughout, $k$ is an algebraic closed field with $ch(k)\not=2$. All the algebras without other specified are over $k$.

\begin{theorem}\thlabel{2.1}\cite{Gab}
The following algebras are pairwise non-isomorphic except pairs within the family $(18;\lambda|0)$ where $(18;\lambda_1|0)\cong (18;\lambda_2|0)$ if and only if $\lambda_1=\lambda_2$ or $\lambda_1\lambda_2=1$.

$\begin{array}{llll}
(1|0) & k \times k \times k \times k, &
(2|0) & k \times k \times k[X]/(X^{2}),\\
(3|0) & k[X]/(X^{2}) \times k[Y]/(Y^{2}),&
(4|0) & k \times k[X]/(X^{3}),\\
(5|0) & k[X]/(X^{4}),&
(6|0) & k \times k[X,Y]/(X,Y)^{2},\\
(7|0) & k[X,Y]/(X^{2},Y^{2}),&
(8|0) & k[X,Y]/(X^{3},XY,Y^{2}),\\
(9|0) & k[X,Y,Z]/(X,Y,Z)^{2},&
(10|0) & M_2,\\
{(11|0)} &  {\scriptsize\left\{\left.\begin{pmatrix}
a & 0 & 0 & 0\\
0 & a & 0 & d\\
c & 0 & b & 0\\
0 & 0 & 0 & b
\end{pmatrix}\right| a,b,c,d \in k\right\}},  &
{(12|0)} & \wedge k^{2},  \\
{(13|0)} &  k \times {\scriptsize\begin{pmatrix}
k & k \\
0 & k
\end{pmatrix}=\left\{\left.(a, \begin{pmatrix}
b & c \\
0 & d
\end{pmatrix})\right| a,b,c,d \in k\right\}}, &
{(14|0)} &  {\scriptsize\left\{\left.\begin{pmatrix}
a & 0 & 0 \\
c & a & 0 \\
d & 0 & b
\end{pmatrix}\right| a,b,c,d \in k\right\}},  \\
{(15|0)} &  {\scriptsize\left\{\left.\begin{pmatrix}
a & c & d \\
0 & a & 0 \\
0 & 0 & b
\end{pmatrix}\right| a,b,c,d \in k\right\}},  &
{(16|0)} & k \langle X,Y \rangle /(X^{2},Y^{2},YX),  \\
{(17|0)} & {\scriptsize\left\{\left.\begin{pmatrix}
a & 0 & 0 \\
0 & a & 0 \\
 c & d & b
 \end{pmatrix}\right| a,b,c,d \in k\right\}}, && \\
{(18;\lambda|0)} & k \langle X,Y \rangle
/(X^{2},Y^{2},YX-\lambda XY), \ \ \lambda \neq -1,0,1, &&\\
(19|0) & k \langle X,Y \rangle / (Y^{2},X^{2}+YX,XY+YX).&&\\
\end{array}$

\end{theorem}

\begin{theorem}\cite[Thm 3.1]{ACZ} Suppose $A$ is a  superalgebra with $\dim A_0 =3$ and $\dim A_1 = 1$. Then $A$ is isomorphic to one of the  superalgebras in the following pairwise non-isomorphic  families:

$\begin{array}{ll}
 (1|1)  &  k \times k \times k \times k: \\
&  (1|1)_0=k(1,1,1,1) \oplus k(1,0,0,0) \oplus k(0,0,1,1),
 (1|1)_1=k(0,0,1,-1) ,\\
 (2|1)  &  k \times k \times k[X]/(X^{2}): \\
&  (2|1)_0=k(1,1,1)\oplus k(1,0,0)\oplus k(0,1,0),
  (2|1)_1=k(0,0,X) ,\\
  (2|2)  &  (2|2)_0=k(1,1,1)\oplus k(1,1,0)\oplus k(0,0,X),
  (2|2)_1=k(1,-1,0) ,\\
 (3|1) &   k[X]/(X^{2}) \times k[Y]/(Y^{2}): \\
 &  (3|1)_0=k(1,1)\oplus k(1,0)\oplus k(X,0),\
    (3|1)_1=k(0,Y) ,\\
 (4|1)  &   k \times k[X]/(X^{3}): \\
 &  (4|1)_0=k(1,1)\oplus k(1,0)\oplus k(0,X^{2}),\
  (4|1)_1=k(0,X) ,\\
 (6|1) &   k \times k[X,Y]/(X,Y)^{2}: \\
 &  (6|1)_0=k(1,1)\oplus k(1,0)\oplus k(0,X),\
 (6|1)_1=k(0,Y) ,\\
 (7|1)  &   k[X,Y]/(X^2,Y^2): \\
 &  (7|1)_0=k1\oplus k(X+Y) \oplus kXY,\
 (7|1)_1=k(X-Y) ,\\
 (8|1)  &    k[X,Y]/(X^{3},XY,Y^{2}): \\
 &  (8|1)_0=k1\oplus kX\oplus kX^{2}, \
 (8|1)_1=kY ,\\
 (8|2) &  (8|2)_0=k1\oplus kX^{2}\oplus kY,\
 (8|2)_1=kX ,\\
 (9|1) &    k[X,Y,Z]/(X,Y,Z)^{2}: \\
 &  (9|1)_0=k1\oplus kX\oplus kY, \  (9|1)_1=kZ ,\\
  (11|1) &   {\scriptsize\left\{\left.\begin{pmatrix}
a & 0 & 0 & 0\\
0 & a & 0 & d\\
c & 0 & b & 0\\
0 & 0 & 0 & b
\end{pmatrix}\right| a,b,c,d \in k\right\}:} \\
 & \  (11|1)_0=k{\scriptsize\begin{pmatrix}
1 & 0 & 0 & 0\\
0 & 1 & 0 & 0\\
0 & 0 & 1 & 0\\
0 & 0 & 0 & 1\\
\end{pmatrix}}
\oplus k{\scriptsize\begin{pmatrix}
1 & 0 & 0 & 0\\
0 & 1 & 0 & 0\\
0 & 0 & 0 & 0\\
0 & 0 & 0 & 0\\
\end{pmatrix}}
\oplus k{\scriptsize\begin{pmatrix}
0 & 0 & 0 & 0\\
0 & 0 & 0 & 1\\
0 & 0 & 0 & 0\\
0 & 0 & 0 & 0\\
\end{pmatrix}} \\
 & \   (11|1)_1=k{\scriptsize\begin{pmatrix}
0 & 0 & 0 & 0\\
0 & 0 & 0 & 0\\
1 & 0 & 0 & 0\\
0 & 0 & 0 & 0\\
\end{pmatrix}} ,\\
 (13|1)  &   {\scriptsize k\times
\begin{pmatrix}
k & k \\
0 & k
\end{pmatrix}}=
{\scriptsize\left\{\left(a, \begin{pmatrix}
b & c\\
0 & d\\
\end{pmatrix})\right| a,b,c,d \in k\right\}}: \\
 & \  (13|1)_0=k{\scriptsize\left(1, \begin{pmatrix}
1 & 0 \\
0 & 1 \\
\end{pmatrix}\right)}
\oplus k{\scriptsize\left(0, \begin{pmatrix}
1 & 0 \\
0 & 0 \\
\end{pmatrix}\right)}
\oplus k{\scriptsize\left(0, \begin{pmatrix}
0 & 0 \\
0 & 1 \\
\end{pmatrix}\right)} \\
 & \   (13|1)_1= k{\scriptsize
\left(0, \begin{pmatrix}
 0 & 1 \\
 0 & 0 \\
 \end{pmatrix}\right)} ,\\
 (14|1)&   {\scriptsize\left\{\left.\begin{pmatrix}
a & 0 & 0 \\
c & a & 0 \\
d & 0 & b \\
\end{pmatrix}\right| a,b,c,d \in k\right\}}: \\
 & \  (14|1)_0=k{\scriptsize\begin{pmatrix}
1 & 0 & 0\\
0 & 1 & 0\\
0 & 0 & 1\\
\end{pmatrix}}
\oplus k{\scriptsize\begin{pmatrix}
1 & 0 & 0 \\
0 & 1 & 0 \\
0 & 0 & 0 \\
\end{pmatrix}}
\oplus k{\scriptsize\begin{pmatrix}
0 & 0 & 0 \\
0 & 0 & 0 \\
1 & 0 & 0 \\
\end{pmatrix}}
\end{array}$

$\begin{array}{ll}
 & \   (14|1)_1=k{\scriptsize\begin{pmatrix}
0 & 0 & 0 \\
1 & 0 & 0 \\
0 & 0 & 0 \\
\end{pmatrix}} ,\\
(14|2) & \  (14|2)_0=k{\scriptsize\begin{pmatrix}
1 & 0 & 0 \\
0 & 1 & 0 \\
0 & 0 & 1 \\
\end{pmatrix}}
\oplus k{\scriptsize\begin{pmatrix}
1 & 0 & 0 \\
0 & 1 & 0 \\
0 & 0 & 0 \\
\end{pmatrix}}
\oplus k{\scriptsize\begin{pmatrix}
0 & 0 & 0 \\
1 & 0 & 0 \\
0 & 0 & 0 \\
\end{pmatrix}}  \\
 & \   (14|2)_1=k{\scriptsize\begin{pmatrix}
0 & 0 & 0 \\
0 & 0 & 0 \\
1 & 0 & 0 \\
\end{pmatrix}} ,\\
\end{array}$

$\begin{array}{ll}
 (15|1) &  {\scriptsize\left\{\left.\begin{pmatrix}
a & c & d \\
0 & a & 0 \\
0 & 0 & b \\
\end{pmatrix}\right| a,b,c,d \in k\right\}}: \\
 &  (15|1)_0=k{\scriptsize\begin{pmatrix}
1 & 0 & 0 \\
0 & 1 & 0 \\
0 & 0 & 1 \\
\end{pmatrix}}
\oplus k{\scriptsize\begin{pmatrix}
1 & 0 & 0 \\
0 & 1 & 0 \\
0 & 0 & 0 \\
\end{pmatrix}}
\oplus k{\scriptsize\begin{pmatrix}
0 & 0 & 1 \\
0 & 0 & 0 \\
0 & 0 & 0 \\
\end{pmatrix}}  \\
 & \   (15|1)_1=k{\scriptsize\begin{pmatrix}
0 & 1 & 0 \\
0 & 0 & 0 \\
0 & 0 & 0 \\
\end{pmatrix}} ,\\
(15|2) &  (15|2)_0=k{\scriptsize\begin{pmatrix}
1 & 0 & 0 \\
0 & 1 & 0 \\
0 & 0 & 1 \\
\end{pmatrix}}
\oplus k{\scriptsize\begin{pmatrix}
1 & 0 & 0 \\
0 & 1 & 0 \\
0 & 0 & 0 \\
\end{pmatrix}}
\oplus k{\scriptsize\begin{pmatrix}
0 & 1 & 0 \\
0 & 0 & 0 \\
0 & 0 & 0 \\
\end{pmatrix}}  \\
 &   (15|2)_1=k{\scriptsize\begin{pmatrix}
0 & 0 & 1 \\
0 & 0 & 0 \\
0 & 0 & 0 \\
\end{pmatrix}} ,\\
 (17|1)&   {\scriptsize\left\{\left.\begin{pmatrix}
a & 0 & 0 \\
0 & a & 0 \\
c & d & b \\
\end{pmatrix}\right| a,b,c,d \in k\right\}}: \\
 &   (17|1)_0=k{\scriptsize\begin{pmatrix}
1 & 0 & 0 \\
0 & 1 & 0 \\
0 & 0 & 1 \\
\end{pmatrix}}
\oplus k{\scriptsize\begin{pmatrix}
1 & 0 & 0 \\
0 & 1 & 0 \\
0 & 0 & 0 \\
\end{pmatrix}}
\oplus k{\scriptsize\begin{pmatrix}
0 & 0 & 0 \\
0 & 0 & 0 \\
1 & 0 & 0 \\
\end{pmatrix}}  \\
 &    (17|1)_1=k{\scriptsize\begin{pmatrix}
0 & 0 & 0 \\
0 & 0 & 0 \\
0 & 1 & 0 \\
\end{pmatrix}}.
\end{array}$
\end{theorem}

\begin{theorem}\thlabel{4.1} \cite[Thm 4.1]{ACZ}  Suppose $A$ is a superalgebra with  $\dim A_0=\dim A_1=2$. Then $A$ is isomorphic to one of the superalgebras in the following pairwise non-isomorphic families:

$\begin{array}{ll}
  (1|2)   &   k \times k \times k \times k:  \\
&   (1|2)_0=k(1,1,1,1)\oplus k(1,1,0,0)
  \  \mathrm{and}\    (1|2)_1=k(1,-1,0,0)\oplus k(0,0,1,-1),  \\
  (2|3)   &   k \times k \times k[X]/(X^{2}):  \\
 &    (2|3)_0=k(1,1,1)\oplus k(1,1,0)
  \  \mathrm{and}\    (2|3)_1=k(1,-1,0)\oplus k(0,0,X), \\
  (3|2)   &   k[X]/(X^{2}) \times k[Y]/(Y^{2}):  \\
 &    (3|2)_0=k(1,1)\oplus k(1,0)     \  \mathrm{and}\
  (3|2)_1=k(X,0)\oplus k(0,Y),  \\
  (3|3)    &    (3|3)_0=k(1,1)\oplus k(X,Y)     \  \mathrm{and}\
  (3|3)_1=k(1,-1)\oplus k(X,-Y),  \\
  (5|1)   &   k[X]/(X^{4}):  \\
 &    (5|1)_0=k1\oplus kX^{2}     \  \mathrm{and}\
  (5|1)_1=kX\oplus kX^{3},  \\
  (6|2)   &   k \times k[X,Y]/(X,Y)^{2}:  \\
 &    (6|2)_0=k(1,1)\oplus k(1,0)     \  \mathrm{and}\
  (6|2)_1=k(0,X)\oplus k(0,Y),  \\
  (7|2)   &   k[X,Y]/(X^{2},Y^{2}):  \\
 &    (7|2)_0=k1\oplus kX     \  \mathrm{and}\    (7|2)_1=kY\oplus kXY,  \\
(7|3) & (7|3)_0=k1+kXY\ \mathrm{and}\ (7|3)_1=kX+kY.\\
{(8|3)} & k[X,Y]/(X^{3},XY,Y^{2}):\\
& (8|3)_0=k1+kX^{2}\ \mathrm{and}\ (8|3)_1=kX+kY.\\
  (9|2)   &   k[X,Y,Z]/(X,Y,Z)^{2},  \\
 &    (9|2)_0=k1\oplus kX     \  \mathrm{and}\    (9|2)_1=kY\oplus kZ,  \\
  (10|1)   &   M_2:  \\
 &    (10|1)_0=k{\scriptsize\begin{pmatrix}
1 & 0\\
0 & 1
\end{pmatrix}}\oplus k{\scriptsize\begin{pmatrix}
1 & 0\\
0 & 0
\end{pmatrix}}     \  \mathrm{and}\    (10|1)_1=k{\scriptsize\begin{pmatrix}
0 & 1\\
0 & 0
\end{pmatrix}}\oplus k{\scriptsize\begin{pmatrix}
0 & 0\\
1 & 0
\end{pmatrix}},  \\
(11|2)   &    {\scriptsize\left\{\left.\begin{pmatrix}
a & 0 & 0 & 0\\
0 & a & 0 & d\\
c & 0 & b & 0\\
0 & 0 & 0 & b
\end{pmatrix}\right| a,b,c,d \in k\right\}}:  \\
 &    (11|2)_0=k{\scriptsize\begin{pmatrix}
1 & 0 & 0 & 0\\
0 & 1 & 0 & 0\\
0 & 0 & 1 & 0\\
0 & 0 & 0 & 1
\end{pmatrix}}\oplus k{\scriptsize\begin{pmatrix}
1 & 0 & 0 & 0\\
0 & 1 & 0 & 0\\
0 & 0 & 0 & 0\\
0 & 0 & 0 & 0
\end{pmatrix}}   \  \mathrm{and} \\
 &     (11|2)_1=k{\scriptsize\begin{pmatrix}
0 & 0 & 0 & 0\\
0 & 0 & 0 & 1\\
0 & 0 & 0 & 0\\
0 & 0 & 0 & 0
\end{pmatrix}}\oplus k{\scriptsize\begin{pmatrix}
0 & 0 & 0 & 0\\
0 & 0 & 0 & 0\\
1 & 0 & 0 & 0\\
0 & 0 & 0 & 0
\end{pmatrix}},  \\
\end{array}$

$\begin{array}{ll}
  (11|3)    &    (11|3)_0=k{\scriptsize\begin{pmatrix}
1 & 0 & 0 & 0\\
0 & 1 & 0 & 0\\
0 & 0 & 1 & 0\\
0 & 0 & 0 & 1
\end{pmatrix}}\oplus k{\scriptsize\begin{pmatrix}
0 & 0 & 0 & 0\\
0 & 0 & 0 & 1\\
1 & 0 & 0 & 0\\
0 & 0 & 0 & 0
\end{pmatrix}}   \  \mathrm{and} \\
 &    \    (11|3)_1=k{\scriptsize\begin{pmatrix}
1 & 0 & 0 & 0\\
0 & 1 & 0 & 0\\
0 & 0 & -1 & 0\\
0 & 0 & 0 & -1
\end{pmatrix}}\oplus k{\scriptsize\begin{pmatrix}
0 & 0 & 0 & 0\\
0 & 0 & 0 & -1\\
1 & 0 & 0 & 0\\
0 & 0 & 0 & 0
\end{pmatrix}},  \\
  (12|1)   &    \wedge k^{2}\cong
k\langle X, Y\rangle/(X^2, Y^2, XY+YX):  \\
 &    (12|1)_0=k1\oplus kX     \  \mathrm{and}\
(12|1)_1=kY\oplus kXY,  \\
(12|2) & (12|2)_0=k1+kXY\ \mathrm{and}\ (12|2)_1=kX+kY.\\
(14|3)   &    {\scriptsize\left\{\left.\begin{pmatrix}
a & 0 & 0 \\
c & a & 0 \\
d & 0 & b
\end{pmatrix}\right| a,b,c,d \in k\right\}}:  \\
 &    (14|3)_0=k{\scriptsize\begin{pmatrix}
1 & 0 & 0 \\
0 & 1 & 0 \\
0 & 0 & 1
\end{pmatrix}}\oplus k{\scriptsize\begin{pmatrix}
1 & 0 & 0 \\
0 & 1 & 0 \\
0 & 0 & 0
\end{pmatrix}}   \  \mathrm{and} \\
 &       (14|3)_1=k{\scriptsize\begin{pmatrix}
0 & 0 & 0 \\
1 & 0 & 0 \\
0 & 0 & 0
\end{pmatrix}}\oplus k{\scriptsize\begin{pmatrix}
0 & 0 & 0 \\
0 & 0 & 0 \\
1 & 0 & 0
\end{pmatrix}},  \\
  (15|3)   &    {\scriptsize\left\{\left.\begin{pmatrix}
a & c & d \\
0 & a & 0 \\
0 & 0 & b
\end{pmatrix}\right| a,b,c,d \in k\right\}}:  \\
 &    (15|3)_0=k{\scriptsize\begin{pmatrix}
1 & 0 & 0 \\
0 & 1 & 0 \\
0 & 0 & 1
\end{pmatrix}}\oplus k{\scriptsize\begin{pmatrix}
1 & 0 & 0 \\
0 & 1 & 0 \\
0 & 0 & 0
\end{pmatrix}}   \  \mathrm{and} \\
 &        (15|3)_1=k{\scriptsize\begin{pmatrix}
0 & 1& 0 \\
0 & 0 & 0 \\
0 & 0 & 0
\end{pmatrix}}\oplus k{\scriptsize\begin{pmatrix}
0 & 0 & 1 \\
0 & 0 & 0 \\
0 & 0 & 0
\end{pmatrix}}  ,\\
\end{array}$

$\begin{array}{ll}
(16|1)   &    k \langle X,Y \rangle
/(X^{2},Y^{2},YX):  \\
 &    (16|1)_0=k1\oplus kX     \  \mathrm{and}\
  (16|1)_1=kY\oplus kXY,  \\
  (16|2)    &    (16|2)_0=k1\oplus kY     \  \mathrm{and}\
  (16|2)_1=kX\oplus kXY,  \\
 (16|3) & (16|3)_0=k1+kXY\ \mathrm{and} (16|3)_1=kX+kY.\\
 \end{array}$

$\begin{array}{ll}
  (17|2)   &    {\scriptsize\left\{\left.\begin{pmatrix}
a & 0 & 0 \\
0 & a & 0 \\
c & d & b
\end{pmatrix}\right| a,b,c,d \in k\right\}}:  \\
   &    (17|2)_0=k{\scriptsize\begin{pmatrix}
1 & 0 & 0 \\
0 & 1 & 0 \\
0 & 0 & 1
\end{pmatrix}}\oplus k{\scriptsize\begin{pmatrix}
1 & 0 & 0 \\
0 & 1 & 0 \\
0 & 0 & 0
\end{pmatrix}}  \  \mathrm{and}  \\
 &        (17|2)_1=k{\scriptsize\begin{pmatrix}
0 & 0 & 0 \\
0 & 0 & 0 \\
1 & 0 & 0
\end{pmatrix}}\oplus k{\scriptsize\begin{pmatrix}
0 & 0 & 0 \\
0 & 0 & 0 \\
0 & 1 & 0
\end{pmatrix}},  \\
\end{array}$

$\begin{array}{ll}
  (18;\lambda|1)  &   k\langle X,Y\rangle
/(X^{2},Y^{2},YX-\lambda XY)  ,   \  \mathrm{where}\    \lambda\in k\   with
  \lambda \neq -1,0,1\ :  \\
 &   (18;\lambda|1)_0=k1\oplus kX     \  \mathrm{and}\
  (18;\lambda|1)_1=kY\oplus kXY.  \\
 (18;\l|2) & (18;\l|2)_0=k1+kXY\ \mathrm{and}\ (18;\l|2)_1=kX+kY.\\
{(19|1)} & k\langle X,Y\rangle/(Y^{2},X^{2}+YX,YX+XY):\\
& (19|1)_0=k1+kXY\ \mathrm{and}\ (19|1)_1=kX+kY.
\end{array}$

Moreover, $(18;\l|2)\cong(18;\l'|2)$ if and only if $\l'=\l$ or
$\l\l'=1$.
\end{theorem}

There exists only one 4-dimensional superalgebra $A=k\oplus I$ with $A_0=k$ and $A_1=I^2=0$. We denote it by $(9|3)$ as its underlying algebra is isomorphic to $(9)$.

\begin{theorem}\thlabel{5.1}
(Algebraic classification of $4$-dimensional graded algebras)\\
Assume that $k$ is an algebraically closed field and that $\ch(k) \neq 2$.
Let $A$ be a 4-dimensional superalgebra. Then $A$
is isomorphic to one of the following superalgebras.
Moreover each pair of listed superalgebras is non-isomorphic except the superalgebras within the same family $(18;\l|i)$, where $(18;\l|i)\cong(18;\l'|i)$ if and only if $\l'=\l$ or $\l\l'=1$, $i=0,1,2$.

$(1):$\ \ \ \ \ \ $(1|0),$  $(1|1)$,  $(1|2)$,\\
$(2):$\ \ \ \ \ \ $(2|0)$, $(2|1)$, $(2|2)$, $(2|3)$,\\
$(3):$\ \ \ \ \ \ $(3|0)$, $(3|1)$, $(3|2)$, $(3|3)$,\\
$(4):$\ \ \ \ \ \ $(4|0)$, $(4|1)$,\\
$(5):$\ \ \ \ \ \ $(5|0)$, $(5|1)$,\\
$(6):$\ \ \ \ \ \ $(6|0)$, $(6|1)$, $(6|2)$,\\
$(7):$\ \ \ \ \ \ $(7|0)$, $(7|1)$, $(7|2)$, $(7|3)$,\\
$(8):$\ \ \ \ \ \ $(8|0)$, $(8|1)$, $(8|2)$, $(8|3)$,\\
$(9):$\ \ \ \ \ \ $(9|0)$, $(9|1)$, $(9|2)$, $(9|3)$,\\
$(10):$\ \ \ \ \  $(10|0)$, $(10|1)$,\\
$(11):$\ \ \  \ \ $(11|0)$, $(11|1)$, $(11|2)$, $(11|3)$,\\
$(12):$\ \ \  \ \ $(12|0)$, $(12|1)$, $(12|2)$,\\
$(13):$\ \ \  \ \ $(13|0)$, $(13|1)$,\\
$(14):$\ \ \  \ \ $(14|0)$, $(14|1)$, $(14|2)$, $(14|3)$,\\
$(15):$\ \ \  \ \ $(15|0)$, $(15|1)$, $(15|2)$, $(15|3)$,\\
$(16):$\ \ \  \ \ $(16|0)$, $(16|1)$, $(16|2)$, $(16|3)$,\\
$(17):$\ \ \  \ \ $(17|0)$, $(17|1)$, $(17|2)$,\\
$(18;\l):$\mbox{\hspace{0.3cm}}$(18;\l|0)$, $(18;\l|1)$, $(18;\l|2)$,
where $\l\in k$ with $\l\neq 1, 0, -1$,\\
$(19):$\ \ \ \ \ $(19|0)$, $(19|1)$.\\
\end{theorem}

\section{The variety $\Salg_n$ and its properties} \selabel{3.2}

In this section we introduce the variety Salg$_n$ of $n$-dimensional superalgebras. Let $A=A_0 \oplus A_1$ be a superalgebra $A=A_0 \oplus A_1$. The $\mathbb{Z}_2$-grading of $A$ induces an involution given by $\s(a_0+a_1)=a_0-a_1$ where $a_i \in A_i$. Conversely, any algebra involution $\sigma$ of $A$ induces a $\mathbb{Z}_2$-grading on $A$, that is, $A=A_0\oplus A_1$ with $A_0=\{a\in A\mid \sigma(a)=a\}$ and $A_1=\{a\in A\mid \sigma(a)=-a\}$. Thus we can identify a superalgebra $A$ with an algebra $A$ with an involution $\s$, denoted $(A,\s)$.

Let $(A,\s)$ be an $n$-dimensional superalgebra and $\{e_1,e_2,\cdots, e_n\}$ be a basis of $A$. The (unitary associative) algebra structure on vector space $A$ gives rise to a set of structure constants $(\a^k_{ij}) \in \A^{n^3}$ determined by the multiplication of basis vectors so that
$$e_ie_j=\sum_{k=1}^n \a^k_{ij}e_k.$$
The involution $\s$ on $A$ may be also described by a set of constants $(\g^j_i) \in \A^{n^2}$ satisfying $\s(e_i)=\sum_{j=1}^n\g^j_ie_j$. It follows that to each superalgebra, $(A,\s)$, we can associate a set of augmented {\it structure constants} \index{structure constants!of a superalgebra} $(\a^k_{ij},\g^j_i) \in \A^{n^3+n^2}$, where $(\a^k_{ij})$ are the structure constants determined by the algebra structure of $A$ and $(\g^j_i)$ the constants determined by the $\mathbb{Z}_2$-grading in the above manner.  However it is not true that an arbitrary set of augmented structure constants can give rise to a superalgebra. The structure constants must obey certain relations to reflect how we define a superalgebra.

As a superalgebra $(A,\s)$ must in particular be a unitary associative algebra, we have a multiplicative identity which we always take to be the first element of our basis, $e_1$. Then to be a unitary associative algebra we have the following conditions:

\begin{eqnarray*}
& e_1e_i=e_i\\
& e_ie_1=e_i\\
& (e_ie_j)e_k=e_i(e_je_k)
\end{eqnarray*}

Which translate into the following relations amongst the structure constants:

\begin{eqnarray}\eqnlabel{3.1}
& \a^j_{1i}-\d^j_i=0\\
& \a^j_{i1}-\d^j_i=0\\
& \sum_{l=1}^n(\a^l_{ij}\a^m_{lk}-\a^m_{il}\a^l_{jk})=0
\end{eqnarray}

For $\s$ to be an algebra involution means that:

\begin{eqnarray*}
& \s(e_1)=e_1\\
& \s(e_ie_j)=\s(e_i)\s(e_j)\\
& \s^2(e_i)=e_i
\end{eqnarray*}

These become the following relations in terms of the structure constants:

\begin{eqnarray}\eqnlabel{3.2}
& \g^j_1-\d^j_1=0\\
& \sum_{k=1}^n\a^k_{ij}\g^m_k-\sum_{k,l=1}^n \g^k_i\g^l_j\a^m_{kl}=0\\
& \sum_{j=1}^n\g^j_i\g^k_j-\d^k_i=0
\end{eqnarray}

It is precisely those structure constants obeying the relations  (1)-(6) given above which give rise to superalgebras.

\begin{definition}\delabel{2.1} \index{$\Salg_n$}
The equations (1)-(6) given above cut out a variety in $\A^{n^3+n^2}$ which we shall call $\Salg_n$ --- the  variety of $n$-dimensional superalgebras.
\end{definition}

In the rest of this paper we will study the geometry of $\Salg_n$. The geometry of $\Salg_n$ is influenced by that of $\Alg_n$, but $\Salg_n$ has a richer geometrical structure.

\begin{definition}\delabel{2.2} \index{$\SA_n$}
We define $\SA_n$ -- the variety of $n$-dimensional superalgebras not requiring existence of a unit --- to be the subvariety of $\A^{n^3+n^2}$ cut out by equations (3),(5) and (6).
\end{definition}

One checks that if $A$ is a unitary algebra and $\s:A\rightarrow A$ satisfies $\s(xy)=\s(x)\s(y)$ and $\s^2 = \id_A$ then $\s(1_A)=1_A$ (This follows from the more general fact that any invertible homomorphism $\s:A \rightarrow B$ between rings with unit must map the identity to the identity, i.e. $\s(1_A)=1_B$), which after a little thought shows that $\Salg_n=\SA_n \cap V(\{\a^j_{1i}-\d^j_i,\a^j_{i1}-\d^j_i\})$. So we obtain the following result:

\begin{lemma}\lelabel{}
$\Salg_n$ is a closed subvariety of $\SA_n$.
\end{lemma}

It is important to notice the way that we have defined $\Salg_n$ --- requiring the identity to be fixed --- is analogous to the way $\Alg_n$ is defined in \cite{Leb}, but is not analogous to the way $\Alg_n$ was defined in \cite{Gab}. We may define $\Salg'_n$ \index{$\Salg'_n$} as an analogue to Alg$_n$ in \cite{Gab}, which is the subset of $\SA_n$ consisting of superalgebras with unit, but not necessarily requiring the unit to be the first element of the basis or even in the basis.  A similar proof to the one given in \cite{CB} shows that
$\Salg'_n$ is an open affine subvariety of $\SA_n$.

Similarly to the situation remarked in \cite{Leb}, since for our definition of $\Salg_n$ we require that the identity be the first element in the basis of any superalgebra, a subgroup $G_n$ of $\GL_n$ acts on $\Salg_n$ (not the full group $\GL_n$ as one may expect). This action is induced by considering what happens to the structure constants when one makes a basis change. As the identity must be the first element in the basis, this means that the first column of the matrix describing the basis change must be $\begin{pmatrix} 1 & 0 & \ldots & 0 \end{pmatrix}^T$ (identifying the given basis $\{e_1=1,e_2,\ldots,e_n\}$ with the standard basis vectors for $k^n$). Hence we can describe $G_n$ for $n\geq 2$ as follows: $G_n = \left\{ \begin{pmatrix} 1 & b^T \\ 0 & \Sigma \\ \end{pmatrix}: \Sigma \in \GL_{n-1}, b \in k^{n-1} \right\}$. Thus the algebraic group $G_n$ is of dimension $n^2-n$.

\begin{remark}\relabel{2.5}
If one so desired, our methods could be modified to study $\Salg'_n$ with the action of $\GL_n$. However, one would hope that the geometry of both spaces are very similar --- in particular we would like the degeneration partial orders induced in each space to coincide (the degeneration partial order will be introduced in \seref{3.3}). We would hope that such properties are intrinsic to the superalgebras and thus not depend on the way in which they are represented by a particular variety. We have not investigated this thoroughly, although in \cite{Leb}, it is remarked that this is the case for the degeneration partial orders in $\Alg_n$ and $\Alg'_n$.
\end{remark}

Let $\L=(\l^j_i) \in G_n$ and $(\nu^j_i)=\L^{-1}$. Then we can describe the action of $G_n$ on $\Salg_n$ as follows: $$\L \cdot (\a^k_{ij},\g^j_i)=(\sum_{l,p,q=1}^n\nu^k_l\a^l_{pq}\l^p_i\l^q_j,\sum_{k,l=1}^n\nu^j_k\g^k_l\l^l_i)=(\a'^k_{ij},\g'^j_i)$$ Firstly, recall that the formula for the inverse of a matrix means that we can express the entries $\nu^j_i$ of the matrix $\L^{-1}$ as a polynomial in the entries $\l^j_i$ of the matrix $\L$ and $1/\det(\L)$. Then the above formula expresses the new structure constants $\a'^k_{ij},\g'^j_i$ in $\Salg_n$ as a polynomial in the old structure constants $\a^k_{ij},\g^j_i$, the entries of the matrix $\L \in G_n$ and $1/\det(\L)$ which has non-vanishing denominator. Hence the action gives us a morphism $G_n \times \Salg_n \rightarrow \Salg_n$. The same reasoning also shows that the transport of structure action on $\Alg_n$ gives a morphism $G_n \times \Alg_n \rightarrow \Alg_n$.

We may refer to the above action of $G_n$ on $\Salg_n$ as the {\bf transport of structure action}. \index{transport of structure action!for superalgebras} However as it is the only action of $G_n$ on $\Salg_n$ considered here, we shall often simply refer to it as the action of $G_n$ on $\Salg_n$.
It is clear that the orbits of $\Salg_n$ under the action of $G_n$ can be identified with the isomorphism classes of $n$-dimensional superalgebras.

For an $n$-dimensional superalgebra $A$, we will sometimes use $G_n \cdot A$ to represent the orbit in $\Salg_n$ which the isomorphism class of $A$ can be identified with. If in some basis the superalgebra $A$ has structure constants $(\a^k_{ij},\g^j_i)$, then $G_n \cdot A=G_n \cdot (\a^k_{ij},\g^j_i)$.

There are two interesting morphisms between $\Salg_n$ and $\Alg_n$. They arise from the following observations: any $n$-dimensional superalgebra may be regarded as an $n$-dimensional algebra and any $n$-dimensional algebra can be endowed with the trivial $\mathbb{Z}_2$-grading making it into an $n$-dimensional superalgebra.

The first morphism: $U:\Salg_n \rightarrow \Alg_n$ is defined by $(\a^k_{ij},\g^j_i) \mapsto (\a^k_{ij})$ is the forgetful map, which forgets the superalgebra structure on $A$ and only remembers the algebra structure on $A$.

The second morphism: $I:\Alg_n \rightarrow \Salg_n$ is defined by $(\a^k_{ij}) \mapsto (\a^k_{ij},\d^j_i)$ where $\d^j_i$ is the Kronecker delta function. This takes an algebra structure on $A$ and endows it with the trivial $\mathbb{Z}_2$-grading making it a superalgebra on $A$.

Notice that the subset of $\Salg_n$ consisting of superalgebras with the trivial $\mathbb{Z}_2$-grading is a closed subset of $\Salg_n$ and is given by $V(\{\g^j_i-\d^j_i\}) \cap \Salg_n$. The morphism $I$ above identifies $\Alg_n$ with this subset. This result is a part of the following proposition.

\begin{proposition}\prlabel{2.8}
The morphisms $U$ and $I$ described above are continuous closed maps. Moreover $I$ provides an isomorphism of $\Alg_n$ with the closed subset of $\Salg_n$ consisting of the superalgebras with the trivial $\mathbb{Z}_2$-grading.
\end{proposition}

We point out that both morphisms $U$ and $I$ are $G_n$-equivariant. That is, for $\L \in G_n$ and $(\a^k_{ij},\g^j_i)\in$ Salg$_n$, we have $U(\L \cdot (\a^k_{ij},\g^j_i))=\L \cdot U((\a^k_{ij},\g^j_i))$ and $I(\L \cdot (a^k_{ij}))=\L \cdot I((\a^k_{ij}))$.
As a consequence of the $G_n$-equivariance of $U$, we obtain the following:

\begin{corollary}\colabel{2.9}
$U\left(\overline{G_n \cdot (\a^k_{ij},\g^j_i)}\right)=\overline{G_n \cdot (\a^k_{ij})}$.
\end{corollary}

Suppose that one has a superalgebra $A$ with $\dim A_0=i$ and $\mathbb{Z}_2$-grading given by the algebra involution $\s$. Now change to a homogeneous basis (say by a linear map represented by the matrix $\L$), which clearly has $\mathbb{Z}_2$-grading $\s'$ given by the linear map represented by the diagonal matrix with $1$ for the first $i$ entries and $-1$ for the last $n-i$ entries. From the above we have $\s'=\L\s \L^{-1}$, so $\s=\L^{-1}\s'\L$.  Thus $\tr(\s)=\tr(\L^{-1}\s'\L)=\tr(\s'\L\L^{-1})=\tr(\s')=i-(n-i)=2i-n$ and $\det(\s)=\det(\L^{-1}\s'\L)=\det(\L^{-1})\det(\s')\det(\L)=\det(\s')=(-1)^{n-i}$.

We now define $\Salg^i_n$ to be the subset of $\Salg_n$ consisting  of the superalgebras $A$ with $\dim A_0=i$. Obviously we have $\Salg_n = \bigcup^n_{i=1}\Salg^i_n$. Hence, from above, {\bf the trace and determinant are constant on} $\Salg^i_n$. It is clear that these subsets must be disjoint. We are interested in when these subsets are also closed. The following lemma gives some sufficient conditions for this to be the case.

Before stating the next couple of results we mention how vital the assumption that $\ch(k) \neq 2$ is to \leref{2.10} and \prref{2.12}. These are very basic results about the geometry of $\Salg_n$ --- the study of $\Salg_n$ over an algebraically closed field $k$ with $\ch(k)=2$ would require new techniques as the proofs of these two results do not work in the case $\ch(k)=2$.

\begin{lemma}\lelabel{2.10}
The sets $\Salg^i_n$ are closed subsets of $\Salg_n$ in the following situations:\\
\begin{tabular}{ll}
(a) & $\ch(k)=p$ and $n \leq 2p$\\
(b) & $\ch(k)=0$ (with no restriction on $n$ in this case)\\
(c) & $n \leq 6$ (for any algebraically closed field $k$ with $\ch(k) \neq 2$)
\end{tabular}
\end{lemma}
\begin{proof}
Define $\mathrm{S}^i_n=V(\{\sum^n_{j=1}\g^j_j-(2i-n), \sum_{\pi}\sgn(\pi)\g^{\pi(1)}_1\ldots\g^{\pi(n)}_n-(-1)^{n-i}\})\cap \Salg_n$ for $i \in \{1,\ldots,n\}$, (where $\sgn(\pi)$ denotes the signature of the permutation $\pi$, and the sum is taken over all permutations of $\{1,\ldots,n\}$). Thus the $\mathrm{S}^i_n$ are closed subsets of $\Salg_n$. From the statements above, it is clear that $\Salg^i_n \subseteq \mathrm{S}^i_n$. The first polynomial $\sum^n_{j=1}\g^j_j$ represents the trace of the $\mathbb{Z}_2$-grading and the second $\sum_{\pi}\sgn(\pi)\g^{\pi(1)}_1\ldots\g^{\pi(n)}_n$ represesnts its determinant.

For the proof of part (a), consider the following. Let $i,j \in \{1, \ldots, n\}, i \neq j$. If $i$ and $j$ differ by $2p$ then both the traces and the determinants for $\Salg^i_n$ and $\mathrm{S}^j_n$ will agree, so $\Salg^i_n \subseteq \mathrm{S}^j_n$. If $i$ and $j$ differ by less than $2p$, then the traces of $\Salg^i_n$ and $\mathrm{S}^j_n$ will differ unless $i$ and $j$ differ by $p$, in which case, since $p$ is odd (remember we are excluding  the case $\ch(k)=2$ throughout this paper) the determinants will differ. Thus $\Salg^i_n$ and $\mathrm{S}^j_n$ are disjoint. From these comments one can see that we have the equality $\Salg^i_n = \mathrm{S}^i_n$ for all $i \in \{1, \ldots, n\}$ if and only if there are no two distinct integers $i,j \in \{1, \ldots, n\}$ which differ by $2p$. One can always be sure that this is condition is met when $n \leq 2p$. This completes the proof of (a).

For part (b), we have $\ch(k)=0$. Here one simply needs to consider the traces on $\Salg^i_n$ and $\mathrm{S}^j_n$, which must differ unless $i=j$, showing that the subsets $\Salg^i_n$ and  $\mathrm{S^j_n}$ are disjoint unless $i=j$, that is $\Salg^i_n=\mathrm{S}^i_n$.

Finally, for part (c) we combine the results of (a) and (b). In the case of positive characteristic $p$, then as $p \geq 3$, from part (a) we know that these subsets are disjoint and closed for $n \leq 6$, while in the case of zero characteristic from part (b) we know that these subsets are disjoint and closed for any $n$. Combine these statements to see that regardless of the characteristic of the field $k$, the subsets $\Salg^i_n$ are all closed subsets when $n \leq 6$.
\end{proof}

\begin{remark}\relabel{2.11}
\leref{2.10} is likely to be general enough for us to use in all cases where determining irreducible components of $\Salg_n$ is currently practical. The irreducible components of $\Alg_n$ have so far only been described for $n \leq 5$ (with some special --- ``rigid" --- components described in the case $n=6$), and finding these irreducible components is a more basic question than finding the irreducible components of $\Salg_n$. However, it is of theoretical interest to determine whether the subsets $\Salg^i_n$ are in fact closed subsets of $\Salg_n$ for all $n$ and any field $k$ with $\ch(k) \neq 2$, or if there is some field $k$ of prime characteristic, $p$, and some integer, $n$, such that the variety $\Salg_n$ over the field $k$ has one of its subsets $\Salg^i_n$ which is not closed. As we shall see, when the $\Salg^i_n$ are closed they form the connected components of $\Salg_n$. Thus it would be interesting to know if the geometry of $\Salg_n$ can change in this manner for some integer, $n$, and field, $k$, of prime characteristic, $p$.
\end{remark}

Using the notation from the proof of \leref{2.10} we have the following situation for the variety $\Salg_7$ over an algebraically closed field of characteristic $3$. $\mathrm{S}^1_7=\mathrm{S}^7_7=V(\{\sum^n_{j=1}\g^j_j-1, \sum_{\pi}\sgn(\pi)\g^{\pi(1)}_1\ldots\g^{\pi(n)}_n-1\}) \cap \Salg_7$. This is the smallest example of where the above lemma may not be applied. While it is clear that $\Salg^1_7$ and $\Salg^7_7$ are disjoint, it may be possible that $\overline{\Salg^1_7}$ and $\Salg^7_7$ have some point in common. (Recall that we remarked earlier that $\Salg^n_n$ is closed --- so $\overline{\Salg^n_n}=\Salg^n_n$ and thus we do know that  $\overline{\Salg^7_7}=\Salg^7_7$ and $\Salg^1_7$ are disjoint).

\begin{proposition}\prlabel{2.12}
$\Salg_n$ is disconnected for $n \geq 2$.
\end{proposition}
\begin{proof}
By the comments above \leref{2.10}, for each superalgebra, the determinant of the $\mathbb{Z}_2$-grading is either $-1$ or $1$. Since $\ch(k)\neq 2$, $-1$ and $1$ are distinct elements of $k$, hence $X_{-1}=V(\{\sum_{\pi}\sgn(\pi)\g^{\pi(1)}_1\ldots\g^{\pi(n)}_n-(-1)\})\cap \Salg_n$ and $X_1=V(\{\sum_{\pi}\sgn(\pi)\g^{\pi(1)}_1\ldots\g^{\pi(n)}_n-1\})\cap \Salg_n$ are disjoint closed subsets whose union is $\Salg_n$. But $X_{-1}=\Salg_n \bs X_1$ and $X_1=\Salg_n \bs X_{-1}$, hence both are open sets too. Thus $\Salg_n$ is a union of two disjoint open subsets. Both subsets are non-empty for $n \geq 2$. Thus for $n \geq 2$, $\Salg_n$ is disconnected.
\end{proof}
From here onwards, we make {\bf the assumption that $\Salg^i_n$ are  closed subsets of} $\Salg_n$.
The main examples which we are interested in are $\Salg_n$ for $n=2,3,4$, and in these cases this assumption is satisfied by \leref{2.10}.

Since some algebras and superalgebras will arise frequently, we shall name them for convenience.
Define $C_n$ to be the algebra $k[X_1, \ldots, X_{n-1}]/(X_1, \ldots, X_{n-1})^2$ and for $i=1, \ldots, n$, let $C_n(i)$ be the superalgebra which has $C_n$ as its underlying algebra and the $\mathbb{Z}_2$-grading is given by $C_n(i)_0=\spn\{1,X_1, \ldots, X_{i-1}\}, \linebreak C_n(i)_1=\spn\{X_i, \ldots,X_{n-1}\}$. It is clear that algebra $C_n$ and the superalgebras $C_n(i)$ for $i=1,\ldots,n$ all have dimension $n$.

The following lemma shows that each superalgebra structure on $C_n$ is isomorphic to one of the $C_n(i)$.

\begin{lemma}\lelabel{2.15}
Consider the algebra $C_n$. There are $n$ distinct isomorphism classes of superalgebras on this algebra, which are $C_n(1), \ldots, C_n(n)$.
\end{lemma}
\begin{proof}
Let $B=B_0\oplus B_1$ be a superalgebra structure on $C_n$ where $\dim B_0=i+1$ with $0 \leq i \leq n-1$ (so $\dim B_1=n-i-1$). Suppose $B_0$ has basis $\{1,u_1,\ldots,u_i\}$ and $B_1$ has basis $\{u_{i+1},\dots,u_{n-1}\}$. There must be scalars such that for $1\leq j \leq n-1$, $u_j=\a_{j1}1+\a_{j2}X_1+\ldots+\a_{jn}X_{n-1}$.

Now let $u'_j=u_j-\a_{j1}1=\a_{j2}X_1+\ldots+\a_{jn}X_{n-1}$. Then $\{1,u'_1,\dots,u'_i\}$ is also a basis for $B_0$.

If $\a_{j1} \neq 0$ for any $i+1 \leq j \leq n-1$ then $u_j=\a_{j1}1+\sum_{i=1}^{n-1}\a_{ji+1}X_i$, so $u_j^2=\a_{j1}^21+2\sum_{i=1}^{n-1}\a_{ji+1}X_i$. Since $u_j^2 \in B_0$ we must have $\sum_{i=1}^{n-1}\a_{ji+1}X_i \in B_0$, say $\sum_{i=1}^{n-1}\a_{ji+1}X_i=\b_11+\sum^i_{k=1}\b_{k+1}u_k$ then $(\b_1+\a_{j1})1+\sum_{k=1}^i\b_{k+1}u_k-u_j=0$, which contradicts the linear independence of the basis. So $\a_{j1}=0$ for all $i+1\leq j \leq n-1$.

It is easy to check that any two of $u'_1, \ldots, u'_i, u_{i+1}, \ldots, u_{n-1}$ have product zero (including a product involving two of the same terms). So we can define a map $\p:B \rightarrow C_n(i+1)$ by $1 \mapsto 1, u'_1 \mapsto X_1, \ldots, u'_i \mapsto X_i, u_{i+1} \mapsto X_{i+1}, \ldots, u_{n-1} \mapsto X_{n-1}$. It is easy to see that this is a bijection, which preserves the algebra structure and $\mathbb{Z}_2$-grading, hence is an isomorphism of superalgebras. Thus a superalgebra structure on $C_n$ must be isomorphic to one of those described in the lemma.

To conclude the proof, we note that the $n$ superalgebra structures given in the lemma are clearly mutually non-isomorphic.
\end{proof}

\bigskip

So for each $i$ there is a unique (up to isomorphism) superalgebra structure $A$ on $k[X_1,\ldots,X_{n-1}]/(X_1,\ldots,X_{n-1})^2$ which has $\dim A_0=i$.

In the case of $n$-dimensional algebras, Gabriel showed that the closed orbit consists of algebras isomorphic to $C_n$. The closed orbits in $\Salg_n$ consist of superalgebras isomorphic to one of the superalgebras $C_n(i)$, as the following Proposition shows.

\begin{proposition}\prlabel{2.16}
There are $n$ closed orbits in $\Salg_n$. They are all disjoint, $C_n(i)$ being the closed orbit in $\Salg^i_n$.
\end{proposition}
\begin{proof}
Suppose $G_n \cdot A$ is a closed orbit, i.e. $\overline{G_n \cdot A}=G_n \cdot A$. As $U(A)$ is an $n$-dimensional algebra, $G_n \cdot U(A)$ is an orbit in $\Alg_n$. Now by \coref{2.9} $\overline{G_n \cdot U(A)}=U(\overline{G_n\cdot A})=U(G_n \cdot A)=G_n \cdot U(A)$. Thus the orbit $G_n \cdot U(A)$ is closed in $\Alg_n$ but then, by the results of \cite{Gab}, $U(A)$ must be isomorphic to $C_n$. That is, $A$ must be isomorphic to a superalgebra structure on $C_n$.

It remains to show that the orbits, $G_n \cdot C_n(i)$, corresponding to the isomorphism classes of the superalgebras $C_n(i)$ are, in fact, closed. Notice that $C_n=U(C_n(i))$ is the algebra structure whose isomorphism class corresponds to the closed orbit in $\Alg_n$. That is, the orbit $G_n \cdot C_n$ is closed in $\Alg_n$ and thus $U^{-1}(G_n \cdot C_n)$ is closed in $\Salg_n$. Now, by assumption, $\Salg^i_n$ are closed disjoint subsets, thus $U^{-1}(G_n \cdot C_n) \cap \Salg^i_n$ is closed. However this set is the orbit $G_n \cdot C_n(i)$ (since \leref{2.15} above showed that all superalgebra structures on algebra $C_n$ with the degree zero component having dimension $i$ are all isomorphic). The result follows.
\end{proof}

\begin{lemma}\lelabel{2.18}
Suppose that $\Salg^i_n$ are closed subsets. Let $A$ be a superalgebra with $\dim A_0=i$. Assume that there is only one isomorphism class of superalgebras on $U(A)$ which has $\dim_0=i$. If the orbit $G_n \cdot U(A)$ is open in $\Alg_n$ then the orbit $G_n \cdot A$ is open in $\Salg_n$.
\end{lemma}
\begin{proof}
Since $\Salg^i_n$ are all disjoint closed subsets by assumption, they are also each open. Now $U^{-1}(G_n \cdot U(A)))$ is the collection of superalgebra structures on $U(A)$. Since $G_n \cdot U(A)$ is open, so too must be $U^{-1}(G_n \cdot U(A))$, by the continuity of $U$. Now by the assumptions made $G_n \cdot A=U^{-1}(G_n \cdot U(A)) \cap \Salg^i_n$. Thus $G_n \cdot A$ is the intersection of two open sets, so it is open itself.
\end{proof}

\begin{example}\exlabel{2.19}
This is indeed the case for several orbits in $\Salg_4$. Using this result and the fact that the orbits of $(1)$ and $(10)$ are open in $\Alg_4$ we discover that the orbits $(1|0),(1|1),(1|2),(10|0)$ and $(10|1)$ are open in $\Salg_4$.
\end{example}

\section{Algebraic groups and their actions} \selabel{3.3}

Recall that an algebraic group $G$ is an algebraic variety which additionally has the structure of a group. That is, the multiplication $\mu:G \times G \rightarrow G$ given by $\mu(x,y)=xy$ and inversion $\i:G \rightarrow G$ given by $\i(x)=x^{-1}$ are morphisms of varieties. An algebraic group is said to be connected if it is irreducible as a variety. The algebra group $G_n$ and $GL_n$ are connected with dimensions $n^2-n$ and $n^2$ respectively.

The following result is well-known, for example see \cite{CB}.

\begin{lemma}\lelabel{3.6}
Let $G$ be a connected algebraic group acting on a variety $X$, then:
\begin{itemize}
\item[(a)] Each orbit $G \cdot x$ is locally closed (i.e. $G \cdot x$ is open in $\overline{G \cdot x}$) and irreducible\\
\item[(b)] $\dim G \cdot x = \dim G - \dim \Stab_G(x)$\\
\item[(c)] $\overline{G \cdot x} \bs G \cdot x$ is a union of orbits of dimension $ < \dim G \cdot x$\\
\end{itemize}
\end{lemma}

Note that in the case of the $G_n$-action on $\Salg_n$, the stabiliser subgroup of a point $(\a^k_{ij},\g^j_i)$ of $\Salg_n$ is the automorphism group of the superalgebra given by the point $(\a^k_{ij},\g^j_i)$.

Whenever we have a connected algebraic group $G$ acting on a variety $X$, we have the idea of degeneration. The action of $G$ on $X$ partitions the variety into equivalence classes under the equivalence relation $x \equiv y \Leftrightarrow \exists \ g \in G$ such that $y=g \cdot x$. The equivalence classes are the $G$-orbits. Because of this, we shall use the notation $[x]=G\cdot x$ for brevity, while stating and proving results about this more general notion of degeneration.

\begin{definition}\delabel{3.8}
We say that $[x]$ {\bf degenerates} to $[y]$ if $y \in \overline{G \cdot x}$ and will write $[x] \rightarrow [y]$.
\end{definition}

It is not difficult to see that $[x]\rightarrow [y]$ if and only if $G \cdot y \subseteq \overline{G \cdot x}$. The latter provides a useful way to visualize the notion of degeneration --- that an orbit is contained in the closure of some other orbit.

By appealing to \leref{3.6} we can show that this idea of degeneration is not only well-defined on the $G$-orbits of $X$, but it also gives rise to a partial order on the $G$-orbits in $X$. We define $[y] \leq_{degr} [x]$ if and only if $[x]$ degenerates to $[y]$. (Note that in some places the degeneration partial order is defined to be the opposite to this. This happens for example in \cite{Zwa}).

Our main interest is in the degeneration of superalgebras and the degeneration partial order on the isomorphism classes of $n$-dimensional superalgebras.

For $n$-dimensional superalgebras $A$ and $B$, if $(\a^k_{ij},\g^j_i) \in G_n \cdot B$ and $(a^k_{ij},\g^j_i) \in \overline{G_n \cdot A}$, then  {\bf $A$ degenerates to $B$} and denote this by $A \rightarrow B$. In some places the terminology {\bf $A$ dominates $B$} is used instead of $A$ degenerates to $B$.  Clearly, whenever $(\a^k_{ij},\g^j_i) \in G_n \cdot A$, then we also have $(\a^k_{ij},\g^j_i) \in \overline{G_n \cdot A}$ since $G_n \cdot A \subseteq \overline{G_n \cdot A}$. A degeneration of this form is referred to as a {\bf trivial degeneration}, \index{degeneration!trivial} any degeneration which is not of this form is called a {\bf non-trivial degeneration}. \index{degeneration!non-trivial}

Intuitively, if the superalgebra $A$ degenerates to the superalgebra $B$ (where $B \ncong A$ that is, this is a proper degeneration) then we think of the orbit $G_n \cdot B$ as consisting of some of those points outside the orbit $G_n \cdot A$, but which are ``close to" some of the points in the orbit $G_n \cdot A$. This is supported by observing that the orbit $G_n \cdot B$ belongs to the boundary of $G_n \cdot A$ (i.e. the set $\overline{G_n \cdot A}\bs G_n \cdot A$) as we shall see in the next section. Another observation supporting this intuition is that some degenerations may be obtained by taking a sequence of points in the orbit $G_n \cdot A$ whose ``limit" lies in the orbit $G_n \cdot B$ (see \coref{4.3}).

It is well-known that
when $G$ is a connected algebraic group acting on a variety $X$, the irreducible components of $X$ are stable under the action of $G$. Thus we have the following.

\begin{corollary}\colabel{3.11}
When $G$ is a connected algebraic group acting on a variety, the irreducible components are closures of a single orbit or closures of an infinite family of orbits.
\end{corollary}
\begin{proof}
We know that irreducible components are $G$-stable. We also know that components are closed, hence each component can be taken to be the closure of a union of orbits. If there are only finitely many orbits in the union, then by using $\overline{A \cup B}=\overline{A} \cup \overline{B}$ we see that the component is not irreducible unless it is the closure of a single orbit. This gives the required statement.
\end{proof}

\bigskip

In the case of the $G_n$ transport of structure action on $\Alg_n$ Flanigan goes further, and in \cite{Flan} proves a result describing algebraic properties of algebras belonging to some infinite family, whose orbits give rise to an irreducible component as described above.

In the following we shall abuse the terminology, and refer to the situation when some structure is contained in the closure of the union of the orbits of an infinite family of orbits, as a degeneration. We see an example of this in $\Alg_4$ in the results of Gabriel, where the structure $(19)$ is contained in the closure of the union of orbits of the family of structures $(18;\l)$. It is important to notice, however, that this is not a degeneration as defined earlier. Similarly, when an infinite family of orbits is contained in another infinite family of structures, we may also wish to refer to this as a degeneration too. We have an example of this given by Mazzola's work on $\Alg_5$ in \cite{Maz}, where the orbits of the infinite family of structures $(35;\l)$ is contained in the closure of the union of the orbits in the infinite family of structures $(13;\l)$. Finally, one may wish to refer to the case where an infinite family of structures is contained in the closure of a single orbit as a degeneration. This idea is less of an abuse of terminolgy than the others mentioned above, however, since we could consider it to be an infinite family of degenerations (in the original sense), one to each of the orbits in the infinite family. Although an abuse of terminology, it is useful to extend the notion of degeneration in this way, as it helps with determining the irreducible components.

\begin{corollary}\colabel{3.12}
When $G$ is a connected algebraic group acting on a variety $X$, we have the following statements regarding the notions of degeneration and irreducible components:

\begin{enumerate}[(a)]
\item{If $[x] \rightarrow [y]$ then $[y]$ belongs to all the irreducible components to which $[x]$ belongs (and possibly more too)}.
\item{If there is no degeneration to $[x]$, then its closure is an irreducible component}.
\item{If $\cup_{\l}[x(\l)]$ is irreducible and there is no degeneration to $\cup_{\l}[x(\l)]$ then its closure is an irreducible component}.
\end{enumerate}
\end{corollary}
\begin{proof}
For part (a) $G \cdot y \subseteq \overline{G \cdot x}$, so that any irreducible component containing $G \cdot x$ must also contain $G \cdot y$.

For parts (b) and (c), consider what happens if $\overline{G \cdot x}$ (respectively \linebreak $\overline{\cup_{\l} G \cdot x(\l)}$) is not an irreducible component. Then, as an irreducible set, it must be contained in some irreducible component implying that $[x]$ (respectively $\cup_{\l} [x(\l)]$) is contained in the closure of an orbit, or in the closure of the union of an infinite family of  orbits. This means that there is a degeneration to $[x]$ (respectively $\cup_{\l} [x(\l)]$), contrary to our assumption.
\end{proof}

\begin{remark}\relabel{3.13}
The above to wonder when a union of a family of orbits is irreducible, so that we may apply part (c) of the above. This might not be true for arbitrary actions of algebraic groups on a variety. However the infinite families which arise in $\Alg_4$ and $\Alg_5$ can be shown to be irreducible. We illustrate this idea using the superalgebras $(18;\l|i)$. Firstly fix $i$ as either $0$, $1$ or $2$. Use the basis $e_1=1,e_2=X,e_3=Y,e_4=XY$ of $(18;\l|i)$. Then for the member of the family with parameter value $\l \neq -1$ we have that the structure constant, $\a^4_{23}=\l$. Hence, using this basis, we obtain a set of points in $\Salg_4$. Call this set $S$ --- one point from each orbit corresponding to a member of the family $(18;\l|i)$. This set of points can be identified with $k\bs \{-1\}$ which is irreducible in $\A^1$ (being the distinguished open $D(x+1)$ of $\A^1$), thus the set of points, $S$, is also irreducible. Now denote by $\p:G_n \times \Salg_n \rightarrow \Salg_n$ the morphism arising from the transport of structure action of $G_n$ on $\Salg_n$. The union of the orbits of $(18;\l|i)$ is given by $\p(G_n \times S)$, which, exactly as remarked prceeding to  \coref{3.12}, is seen to be irreducible. So we have shown that the union of orbits of superalgebras $(18;\l|i)$ for $i=0,1,2$ are irreducible. The infinite families in $\Alg_5$ can be shown to be irreducible in a similar manner.
\end{remark}

\coref{3.12} tells us that the irreducible components are the orbits or infinite families of orbits, which no other orbit or infinite family of orbits degenerates to. So if one knows all degenerations between orbits and infinite families of orbits, then it is a trivial matter to determine the irreducible components. Unfortunately, the problem of determining all these degenerations is usually difficult. The problem of determining the irreducible components is somewhat easier, but can still be difficult too.

\begin{definition}\delabel{3.14} \index{superalgebra!generic}
An $n$-dimensional superalgebra $A$ (respectively, a family of superalgebras $A(\l)$) is called {\bf generic}, if the closure of its orbit in $\Salg_n$ --- $\overline{G_n \cdot A}$ (respectively, the closure of the union of the family of orbits --- $\overline{\bigcup_\l G_n \cdot A(\l)}$), is an irreducible component of $\Salg_n$.
\end{definition}

\begin{remark}\relabel{3.15}
A superalgebra $A$, whose orbit is open, is always generic. Since it must lie in some irreducible component (being an irreducible set by part (a) of \leref{3.6}) and, as an open subset of any irreducible set is dense, we must have that $\overline{G_n \cdot A}$ is the entire component.

However the observations in \coref{3.12} applies more generally and can also aid us in finding the irreducible components. For example, after finding that no algebras degenerate to $(17)$ in $\Alg_4$, by applying the closed continuous map $U$, we discover that no superalgebras can degenerate to any of $(17|i)$ for $i=0,1,2$ in $\Salg_4$. Then, by using the observations given in \coref{3.12}, we see that $(17|i)$ for $i=0,1,2$ give rise to irreducible components of $\Salg_4$, hence these algebras are also generic.
\end{remark}

The next two lemmas of this section are concerned with calculating the dimensions of the orbits in $\Salg_n$. We explain how to read these tables now. Each row corresponds to a different algebra structure and the columns of the table are for different $\mathbb{Z}_2$-gradings on that given underlying algebra structure. Thus the underlying algebra structure of the superalgebra determines which row you look in, and which particular $\mathbb{Z}_2$-grading is used to obtain the given superalgebra structure determines which column you look under. We illustrate this by using an example. To find the dimension of the stabilizer of a point in the orbit of $(3|2)$ we look in the row labelled $(3| \cdot)$ and then look under the column labelled $2$ to see that the dimension of the required stabilizer is $2$.

\begin{lemma}\lelabel{3.16}
The following gives the dimensions of the stabilizers of points in the orbits in $\Salg_4$:
\begin{center}
\begin{tabular}[t]{|*5{c|}}\hline
\multicolumn{5}{|c|}{Stabilizer dimensions}\\\hline
$\cdot$ & 0 & 1 & 2 & 3\\ \hline
$(1|\cdot)$ & 0 & 0 & 0 &\\ \hline
$(2|\cdot)$ & 1 & 1 & 1 & 1\\ \hline
$(3|\cdot)$ & 2 & 2 & 2 & 1\\ \hline
$(4|\cdot)$ & 2 & 1 & & \\ \hline
$(5|\cdot)$ & 3 & 2 & & \\ \hline
$(6|\cdot)$ & 4 & 2 & 4 & \\ \hline
$(7|\cdot)$ & 4 & 2 & 3 & 2\\ \hline
$(8|\cdot)$ & 5 & 3 & 3 & 3\\ \hline
$(9|\cdot)$ & 9 & 5 & 5 & 9\\ \hline
$(10|\cdot)$ & 3 & 1 & & \\ \hline
$(11|\cdot)$ & 4 & 3 & 2 & 2\\ \hline
$(12|\cdot)$ & 6 & 3 & 4 & \\ \hline
$(13|\cdot)$ & 2 & 1 & & \\ \hline
$(14|\cdot)$ & 3 & 3 & 2 & 2\\ \hline
$(15|\cdot)$ & 3 & 3 & 2 & 2\\ \hline
$(16|\cdot)$ & 4 & 3 & 3 & 2\\ \hline
$(17|\cdot)$ & 6 & 3 & 4 & \\ \hline
$(18;\l|\cdot)$ & 4 & 3 & 2 & \\ \hline
$(19|\cdot)$ & 4 & 2 & & \\ \hline
\end{tabular}
\end{center}
\end{lemma}
\begin{proof}
If the point $(\a^k_{ij},\g^j_i)$ is in the orbit, $G_4 \cdot A$, which is identified with the isomorphism class of superalgebra $A$, then $\Stab_{G_4}((\a^k_{ij},\g^j_i)) \cong \Aut(A)$ where $\Aut(A)$ is the group of automorphisms of the superalgebra $A$ as mentioned in the paragraph below \leref{3.6}.

We remark that $\dim \mathrm{PGL}_n(k)=n^2-1$, so that $\dim \mathrm{PGL}_2(k)=2^2-1=3$ (see for example \cite{Har})
\end{proof}

\begin{proposition}\prlabel{3.17}
The following gives the dimensions of the orbits in $\Salg_4$:
\begin{center}
\begin{tabular}[t]{|*5{c|}}\hline
\multicolumn{5}{|c|}{Orbit dimensions}\\\hline
$\cdot$ & 0 & 1 & 2 & 3\\ \hline
$(1|\cdot)$ & 12 & 12 & 12 &\\ \hline
$(2|\cdot)$ & 11 & 11 & 11 & 11\\ \hline
$(3|\cdot)$ & 10 & 10 & 10 & 11\\ \hline
$(4|\cdot)$ & 10 & 11 & & \\ \hline
$(5|\cdot)$ & 9 & 10 & & \\ \hline
$(6|\cdot)$ & 8 & 10 & 8 & \\ \hline
$(7|\cdot)$ & 8 & 10 & 9 & 10\\ \hline
$(8|\cdot)$ & 7 & 9 & 9 & 9\\ \hline
$(9|\cdot)$ & 3 & 7 & 7 & 3\\ \hline
$(10|\cdot)$ & 9 & 11 & & \\ \hline
$(11|\cdot)$ & 8 & 9 & 10 & 10\\ \hline
$(12|\cdot)$ & 6 & 9 & 8 & \\ \hline
$(13|\cdot)$ & 10 & 11 & & \\ \hline
$(14|\cdot)$ & 9 & 9 & 10 & 10\\ \hline
$(15|\cdot)$ & 9 & 9 & 10 & 10\\ \hline
$(16|\cdot)$ & 8 & 9 & 9 & 10\\ \hline
$(17|\cdot)$ & 6 & 9 & 8 & \\ \hline
$(18;\l|\cdot)$ & 8 & 9 & 10 & \\ \hline
$(19|\cdot)$ & 8 & 10 & & \\ \hline
\end{tabular}
\end{center}
\end{proposition}
\begin{proof}
We have calculated the dimensions of the automorphism groups, or equivalently, the dimensions of stabilizers of any point in each orbit in \linebreak \leref{3.16} above. We know that the dimension of $G_4$ is $12$. By using part (b) of \leref{3.6}, we can calculate the dimension of the orbit $G_4 \cdot (\a^k_{ij},\g^j_i)$ by subtracting the dimension of the stabilizer, $\Stab_{G_4}((\a^k_{ij},\g^j_i))$, from the dimension of $G_4$ which is $12$.
\end{proof}

\begin{remark}\relabel{3.18}
We remark that to calculate the dimensions of the orbits in the case where we don't require the identity to be fixed (i.e. the orbits in $\Salg'_4$ and in which case $\GL_4$ acts on this variety) we can subtract the dimensions of the stabilizers found in \leref{3.16} from $16$. If we then compare the dimensions of the orbits of the trivially $\mathbb{Z}_2$-graded superalgebras $(i|0)$ for $i=1, \ldots, 18;\l,19$, thus calculated, with those given by Gabriel in \cite{Gab}, we find that the two sets of numbers do not agree. In fact the orbit dimensions that Gabriel gives are exactly one less than the orbit dimensions we calculate in each case. This is strange. Since Gabriel did not give the proof of these facts in \cite{Gab} it is difficult to find an explanation for this difference. However in Mazzola's paper \cite{Maz} on classifying algebras of dimension five, the orbit dimensions are calculated by subtracting the dimension of the automorphism groups from 25 (25 being the dimension of $\GL_5$) --- this would tend to suggest that our methodology for calculating orbit dimensions is correct.
\end{remark}

\section{Degenerations in $\Salg_n$} \selabel{3.4}

In this section we concern ourselves with conditions determining when a degeneration of superalgebras in $\Salg_n$ can or cannot exist. When looking for conditions for the non-existence of degenerations between a given pair of superalgebras, it would be helpful to have some invariants of the superalgebra which are ``rigid" in the sense that if there is a degeneration of superalgebras from $A$ to $B$, then the superalgebras $A$ and $B$ must have the same value for the invariant. Unfortunately, the only such invariant that we know of is $\dim_0$, the dimension of the trivial degree part. The next best thing is a property of a superalgebra which any degeneration of this superalgebra must inherit, or some property which cannot increase or decrease upon degeneration. Such properties are analogous to those described in \cite[Proposition 2.7]{Gab}, which states, for example, the fact that the dimension of the radical cannot decrease upon degeneration. Later in the section we determine several properties from which any degeneration of a given superalgebra must share.

\begin{lemma}\lelabel{4.1}
Let $\Om: k \rightarrow \Salg_n$ be a polynomial function and $U \subseteq \Salg_n$. If there are infinitely many points of $\Om(k)$ in $U$ then $\Om(k) \subseteq \overline{U}$.
\end{lemma}
\begin{proof}
First, note that we think of $\Om$ as describing a curve in $\Salg_n$. $\overline{U}$ is defined to be the intersection of all closed sets containing $U$. A closed set is the vanishing set of polynomials (intersected with $\Salg_n$), so it is enough to show that any polynomial vanishing on $U$ must also vanish on all of  $\Om(k)$.
By applying the appropriate projections to $\Om$, we may write $\a^k_{ij}=a^k_{ij}(t)$ and $\g^j_i=g^j_i(t)$ (letting the indeterminate be $t$), to describe the coordinates of this curve.

It is standard that $\Om^{-1}(U)=\{t \in k: \Om(t) \in U\}$, but notice that this set gives the $t$ values such that the curve $\Om$ lies inside the set $U$. We consider a polynomial function in $(\a^k_{ij},\g^j_i)$, which vanishes on $U$, $f(\a^k_{ij},\g^j_i)=0$. Since $f$ vanishes on $U$ it must vanish at the points of $\Om(k)$ lying inside $U$. So we have $t \in \Om^{-1}(U) \Rightarrow f(a^k_{ij}(t),g^j_i(t))=0$. Note that $f(a^k_{ij}(t),g^j_i(t))$ is a polynomial in $t$. Suppose the degree $\deg(f(a^k_{ij}(t),g^j_i(t)))=d$.

If $d \geq 1$,  then $f(a^k_{ij}(t),g^j_i(t))=0$ has at most $d$ zeros, which contradicts the fact that we assumed to vanish on all of $\Om(k) \cap U$, which has infinitely many points.
Thus $d = 0$, hence $f(a^k_{ij}(t),g^j_i(t))$ must be a constant. The only way that
$f(a^k_{ij}(t),g^j_i(t))=0$ is satisfied for points in $\Om^{-1}(U)$ is if $f(a^k_{ij}(t),g^j_i(t))$ is the zero polynomial, in which case $f(a^k_{ij}(t),g^j_i(t))=0$ is satisfied for all $t \in k$. This completes the proof.
\end{proof}

Now we consider a practical method for computing degeneration of superalgebras,  called a specialization of superalgebras. This method was first introduced by Gabriel in \cite{Gab}. We formulate it in the form of superalgebras.

\begin{definition}\delabel{4.2} \index{specialization!of superalgebras}
If $A$ and $B$ are $n$-dimensional superalgebras, a  {\bf specialization} of $A$ to $B$ is the following situation: one makes a change of basis in $A$ to a ``variable" basis, i.e. one involving some unknown $t$, such that the point of $\Salg_n$ obtained by structural transport is given by some polynomial functions in $t$ and lies in the orbit of $A$ for $t \neq 0$, yet at $t=0$ lies in the orbit $B$. We think of $B$ as being obtained by a formal limit of the basis change in $A$.
\end{definition}

A specialization of superalgebras $A$ to $B$ is a more restrictive notion than a specialization of algebras, since not only must there be a specialization of the underlying algebras, but also must this occur in such a way that under the specialization. The $\mathbb{Z}_2$-grading on $A$ also tends to the $\mathbb{Z}_2$-grading on $B$. This is usually a non-trivial constraint. So some specializations between algebras may not give rise to specializations of superalgebras on these algebras. Or perhaps one must use different specializations for different superalgebra structures on the same underlying algebra.

With this idea of specialization we obtain a useful corollary of the above lemma.

\begin{corollary}\colabel{4.3}
A specialization of $A$ to $B$ implies that $A$ degenerates to $B$.
\end{corollary}
\begin{proof}
Clearly the specialization gives us a curve $\Om: k \rightarrow \Salg_n$. We let the set $U$ in \leref{4.1} be the orbit $G_n \cdot A$. Now, as $k$ is algebraically closed, it has infinitely many elements. Thus so does $k^*$. Then $\Om(k^*) \subseteq G_n \cdot A$, so $G_n \cdot A$ contains infinitely many elements of $\Om(k)$. Thus we may apply \leref{4.1}. Now note that $\Om(0)$ gives structure constants for a point in the orbit $G_n \cdot B$. Hence, by \leref{4.1} the point in the orbit $G_n \cdot B$ given by $\Om(0)$ lies in the closure of the orbit of $A$ --- this means that $A$ degenerates to $B$.
\end{proof}

\begin{remark}\relabel{4.4}
Let $A$ be a superalgebra with $\dim A_0=i$, in other words $A \in \Salg^i_n$. Suppose the bases of $A_0$ and $A_1$ are given by $\{1,e_2, \ldots,e_i\}$ and $\{e_{i+1}, \ldots, \linebreak e_n\}$ respectively. The specialization described by Gabriel in \cite{Gab} given by $1 \mapsto 1, e_2 \mapsto t e_2, \ldots, e_n \mapsto t e_n$ and letting $t \rightarrow 0$ implies that any algebra degenerates to the algebra $C_n$. This specialization does not alter the $\mathbb{Z}_2$-grading, which implies (by \coref{4.3}) any superalgebra in $\Salg^i_n$ degenerates to the superalgebra $C_n(i)$ in $\Salg^i_n$. Stated another way, the closure of any orbit in $\Salg^i_n$ contains the orbit of the superalgebra $C_n(i)$ in $\Salg^i_n$ (which is the closed orbit in $\Salg^i_n$).
\end{remark}

Earlier in \reref{2.11} we mentioned that $\Salg^i_n$ are the connected components of $\Salg_n$. Using \coref{4.3} above, we can now prove this to be the case.

We know that $\A^m$ is a Noetherian space and we have assumed that $\Salg^i_n$ is a closed subset of $\A^m$ (for $m=n^3+n^2$). Thus $\Salg^i_n$ is a union of a finite number of irreducible components. However, irreducible components are closed and they must all contain the orbit of the superalgebra $C_n(i)$ by the above remark. Hence the irreducible components have a non-empty intersection.
Thus $\Salg^i_n$ is a finite union of its irreducible components which are connected and have non-empty intersection. Thus we have showed the following.

\begin{proposition}\prlabel{4.5}
The set $\{\Salg^i_n\}^n_{i=1}$ are the connected components of $\Salg_n$.
\end{proposition}

Note that we needed to assume that $\{\Salg^i_n\}^n_{i=1}$ are closed subsets of $\Salg_n$ in order to prove \prref{4.5}. In fact
one can actually see that $\{\Salg^i_n\}^n_{i=1}$ are the connected components of $\Salg_n$ if and only if $\{\Salg^i_n\}^n_{i=1}$ are closed subsets. \prref{4.5} shows one of the directions, and for the converse we note that connected components are closed (a fact from General Topology).

Given $n$-dimensional superalgebras $A$ and $B$. To show that $A$ can not degenerate to $B$, it is sufficient to exhibit a closed set in $\Salg_n$ containing the orbit $G_n \cdot A$ which is disjoint from $G_n \cdot B$. Note that if there are two disjoint closed sets in $\Salg_n$ one containing the orbit $G_n \cdot A$ and the other containing the orbit of $G_n \cdot B$, then there cannot be any degenerations between $A$ and $B$. We now look for some necessary conditions  for a degeneration of superalgebras to exist.

\begin{remark}\relabel{4.7}
Suppose that $A$ and $B$ are $n$-dimensional superalgebras. The following conditions are necessary for a degeneration.
\begin{enumerate}[(a)]
\item If $A$ doesn't degenerate to $B$ as algebras, then $A$ cannot degenerate to $B$ as superalgebras. This condition becomes sufficient in case the $\mathbb{Z}_2$-gradings of the superalgebras are trivial.
\item There is no degeneration from $A$ to $B$ unless $\dim A_0 = \dim B_0$.
\item When $n \geq 3$, $\Salg^1_n$ consists only of the closed orbit of the superalgebra $C_n(1)$.  In this case, there is no degenerations in $\Salg^1_n$.
\end{enumerate}
\end{remark}

The above facts follow from considering either the algebra structure or the $\mathbb{Z}_2$-grading in isolation. For some more necessary conditions for the existence of a degeneration we must exploit both the algebra structure and the $\mathbb{Z}_2$-grading simultaneously.

Now we look for closed $G_n$-stable subsets defined by some superalgebraic properties. We need the notion of a {\it upper semicontinuous function}. One may find it, for example in \cite{CB}. Given two topological space $X$.
A function $f: X \rightarrow \mathbb{Z}$ is said to be upper semicontinuous if the set $\{x \in X: f(x) \geq n\}$ is closed in $X$ for all $n \in \mathbb{Z}$.

\begin{lemma}\lelabel{1.49} (\cite[Chapter $1$ \S8 Corollary $3$]{Mum})
If $f:X \rightarrow Y$ is a morphism of varieties, then the function $x \mapsto \dim_x f^{-1}(f(x))$ is upper semicontinuous.
\end{lemma}

If $V$ is a vector space and $W$ a subset of $V$, then $W$ is called a {\it cone in $V$} if $W$ contains the zero vector and is closed under scalar multiplication.
The following lemma can be found in \cite{CB}.

\begin{lemma}\lelabel{1.51}
Suppose $X$ is a variety, $V$ a vector space and we are given subsets $V_x \subseteq V$ for all $x \in X$. Suppose that
\begin{enumerate}[(a)]
\item{each $V_x$ is a cone in $V$},
\item{$\{(x,v):v \in V_x\}$ is closed in $X\times V$}.
\end{enumerate}
Then the map $x \mapsto \dim V_x$ is upper semicontinuous.
\end{lemma}

We have the following facts about dimension.

\begin{lemma}\lelabel{1.41}
\begin{enumerate}[(a)]
\item{For an algebraic set $X$, the dimension of $X$ is equal to the Krull dimension of its coordinate ring $A(X)$}.
\item{The dimension of $\A^n$ is $n$}.
\item{If $U \neq \emptyset$ is open in an irreducible variety $X$, then $\dim U=\dim X$}.
\item{If $X=\bigcup^n_{i=1} U_i$ with the $U_i$ irreducible, then $\dim X= \max_{i \in \{1, \ldots, n\}}\{\dim U_i\}$}.
\item{If $X \subseteq Y$ then $\dim X \leq \dim Y$, moreover if $X$ is closed and $Y$ is irreducible, then $X \subset Y$ implies $\dim X < \dim Y$}.
\end{enumerate}
\end{lemma}

\begin{lemma}\lelabel{4.8}
The following sets are closed in $\Salg_n$:
\begin{itemize}
\item[(a)] {$\{A \in \Salg_n:A_1^2=\{0\}\}$}.
\item[(b)] {$\{A \in \Salg_n:A_0$ is commutative $\}$}.
\end{itemize}
\end{lemma}
\begin{proof}
Recall that we defined superalgebra structures on an $n$-dimensional vector space $V$ with a basis $\{e_1, \ldots, e_n\}$.

For the set in part (a) we assign to a superalgebra $A$ the following subset $W_A=\{v\ot w: v,w \in A_1,vw=0\}$ of $V \ot V$. For the set in part (b) we assign to a superalgebra $A$ the following subset $W'_A=\{v\ot w: v,w \in A_0,vw=wv\}$ of $V \ot V$. It is straightforward to check that these are both cones in $V \ot V$.

Then we may write $v=\sum_{i=1}^n c_ie_i$ and $w=\sum_{i=1}^n d_ie_i$. Now from $v \ot w \neq 0$ it is possible to recover $v$ and $w$ up to scalar multiple. This fact shall cause us no problems, however, since $W_A$ and $W'_A$ are cones in $V \ot V$.

We show now that $\{(A,v \ot w):v,w \in A_1, vw=0\}$ is closed in $\Salg_n \times (V \ot V)$. If $v \ot w =0$ then either $v=0$ or $w=0$, in which case $c_i=0$ for $i=1, \ldots, n$ or $d_i=0$ for $i=1, \ldots, n$. So for $v\ot w \neq 0$, $v \in A_1 \Leftrightarrow \sum_{i=1}^n c_i\g_i^j+c_j=0$ for $j=1, \ldots, n$; $w \in A_1 \Leftrightarrow \sum_{i=1}^n d_i\g_i^j+d_j=0$ for $j=1, \ldots, n$; and $vw=0 \Leftrightarrow \sum_{i,j=1}^n c_id_j\a_{ij}^k=0$ for $1\leq i,j \leq n$. We remark that if coordinates of $v$ and $w$ with respect to the given basis, i.e. $(c_i),(d_i)$, satisfy these equations, then so too must $(\l c_i), (\mu d_i)$ for any $\l, \mu \in k$. Thus it does not  matter that we can only obtain $v$ and $w$ up to scalar multiple. Thus $\{(A,v \ot w):v,w \in A_1, vw=0\}=V(\{c_i\}) \cup V(\{d_i\}) \cup V(\{\sum_{i=1}^n c_i\g_i^j+c_j,\sum_{i=1}^n d_i\g_i^j+d_j,\sum_{i,j=1}^n c_id_j\a_{ij}^k\})$, which is closed in $\Salg_n \times (V \ot V)$.

We show next that $\{(A,v \ot w):v,w \in A_0, vw=wv\}$ is closed in $\Salg_n \times (V \ot V)$. If $v \ot w =0$ then either $v=0$ or $w=0$, in which case $c_i=0$ for $i=1, \ldots, n$ or $d_i=0$ for $i=1, \ldots, n$. So for $v\ot w \neq 0$, $v \in A_0 \Leftrightarrow \sum_{i=1}^n c_i\g_i^j-c_j=0$ for $j=1, \ldots, n$; $w \in A_0 \Leftrightarrow \sum_{i=1}^n d_i\g_i^j-d_j=0$ for $j=1, \ldots, n$; and $vw=wv \Leftrightarrow \sum_{i,j=1}^n c_id_j(\a_{ij}^k-\a_{ji}^k)=0$ for $1\leq i,j \leq n$. Thus $\{(A,v \ot w):v,w \in A_0, vw=wv\}=V(\{c_i\}) \cup V(\{d_i\}) \cup V(\{\sum_{i=1}^n c_i\g_i^j-c_j,\sum_{i=1}^n d_i\g_i^j-d_j,\sum_{i,j=1}^n c_id_j(\a_{ij}^k-\a_{ji}^k)\})$, which is closed in $\Salg_n \times (V \ot V)$.

It follows from \leref{1.51} that the maps $A \mapsto \dim W_A$ and $A \mapsto \dim W'_A$ are upper semicontinuous.
Now since $\Salg_n^i$ are closed subsets of $\Salg_n$ it suffices to show that the sets mentioned in the lemma intersected with $\Salg_n^i$ are closed in $\Salg_n^i$ for each $i=1, \ldots, n$. That is, we may assume $\dim A_0=i$. We note that $W_A \subseteq A_1 \ot A_1$. Now if $A_1^2=0$, then $W_A=A_1 \ot A_1$ which has dimension $(n-i)^2$. If $A_1^2 \neq \{0\}$, then $W_A \subset A_1 \ot A_1$. We can see from the above, that for a given superalgebra $A$, $W_A$ is closed in $V \ot V$, and we note that $A_1 \ot A_1$ is irreducible and has dimension $(n-i)^2$ (as a variety) since it is isomorphic to the $(n-i)^2$-dimensional affine space $\A^{(n-i)^2}$. Thus $\dim W_A < (n-i)^2$ by \leref{1.41}. Therefore the set $\{A \in \Salg^i_n:A_1^2=\{0\}\}=\{A \in \Salg^i_n:\dim W_A \geq (n-i)^2\}$ is a closed set by the upper semicontinuity. This proves part (a).

Similarly $W'_A \subseteq A_0 \ot A_0$, and if $A_0$ is commutative then $W'_A =A_0 \ot A_0$ which has dimension $i^2$. If $A_0$ is not commutative then $W'_A \subset A_0 \ot A_0$ and so similarly as above $\dim W'_A < i^2$ (we just need to note that $W'_A$ is closed and $A_0 \ot A_0$ is irreducible). Thus the set $\{A \in \Salg^i_n:A_0$ is commutative $\}=\{A \in \Salg^i_n:\dim W'_A \geq i^2\}$  is a closed set by the upper semicontinuity. This proves part (b).
\end{proof}

For $\Salg^2_n$ we have other closed subsets. Since $\dim A_0 = 2$, $J(A_0)=\{x \in A_0: x^2=0\}$, notice that this is a vector subspace of $A_0$.

\begin{lemma}\lelabel{4.11}
The following are closed sets in $\Salg^2_n$:
\begin{itemize}
\item[(a)] {$\{A\in \Salg^2_n: \dim J(A_0)=1\}$}.
\item[(b)] {$\{A\in \Salg^2_n: \dim J(A_0)=1,J(A_0)A_1=\{0\}\}$}.
\item[(c)] {$\{A\in \Salg^2_n: \dim J(A_0)=1,A_1J(A_0)=\{0\}\}$}.
\end{itemize}
\end{lemma}
\begin{proof}
We give the proof for subset (b), since the proof for subset (c) is very similar and the proof for subset (a) follows by just simplifying this proof.

For subset (b) we assign to a superalgebra $A$ the subset $W_A=\{v\ot w: v \in A_0, w \in A_1, v^2=0, vw=0\}$ of $V \ot V$. This is clearly a cone. We also note $W_A \subseteq J(A_0) \ot A_1$.
Suppose $v=\sum_{i=1}^n c_i e_i$, $w=\sum_{i=1}^n d_i e_i$. We discover $\{(A,v \ot w): v \ot w \in W_A\}=V(\{c_i\}) \cup V(\{d_i\}) \cup V(\{\sum_{i=1}^n c_i\g_i^j-c_j,\sum_{i=1}^nc_ic_j\a_{ij}^k,\sum_{i=1}^n d_i\g_i^j+d_j,\sum_{i,j=1}^n c_id_j\a_{ij}^k\})$. Which is closed in $\Salg_n \times (V \ot V)$.

So by \leref{1.51}, $A \mapsto \dim W_A$ is an upper semicontinuous map.
Now, if $A \in \{A\in \Salg^2_n: \dim J(A_0)=1,J(A_0)A_1=\{0\}\}$ then $\dim W_A=n-2$.
If $A \notin \{A\in \Salg^2_n: \dim J(A_0)=1,J(A_0)A_1=\{0\}\}$ then either $\dim J(A_0)=0$ in which case $W_A=\{0\}$ and $\dim W_A=0$ or $\dim J(A_0)=1$ and $J(A_0)A_1 \neq \{0\}$ in which case $W_A \subset J(A_0) \ot A_1$. In this case $\dim W_A < n-2$ since $W_A$ is closed, and $J(A_0)\ot A_1 \cong A_1 \cong \A^{n-2}$ as vector spaces, so $J(A_0) \ot A_1$ is an irreducible subset of dimension $n-2$ as an $(n-2)$-dimensional vector subspace $W$ of $\A^n$ with $n > r$ is isomorphic as a variety to $\A^r$. In particular this means that $W$ is irreducible and as a variety has dimension $r$.

Hence $\{A\in \Salg^2_n: \dim J(A_0)=1,J(A_0)A_1=\{0\}\}=\{A \in \Salg^2_n: \dim W_A \geq n-2\}$ which is closed by the upper semi-continuity.
\end{proof}

\bigskip

One can quickly check that if a superalgebra belongs to one of the closed sets described in \leref{4.8},  or \leref{4.11}, then any isomorphic superalgebra must also belong to the same set. Thus these closed sets are stable under the action of $G_n$.

\section{Degenerations in $\Salg_4$} \selabel{3.5}

In this section we are interested in determining when $4$-dimensional superalgebra structures do or do not degenerate to one another. Here we use the results derived in the previous section to help us.

The results of this section give us most of the degenerations in $\Salg_4$. Before giving the degeneration diagrams we shall first explain how to interpret them. We follow this by giving a partial classification theorem for $\Salg_4$ --- we determine twenty irreducible components. There are, however, two other structures which may or may not give rise to irreducible components, and finally we give the details of the degenerations or the non-existence of degenerations, which were shown in the degeneration diagram.

As we shall soon see, there can be no degenerations amongst $4$-dimensional superalgebras $A$ and $B$ with $\dim A_0 \neq \dim B_0$. Thus we can give the degeneration diagram for $\Salg_4$ by giving the degeneration diagrams for each of the connected components $\Salg^i_4$ for $i=1,2,3,4$ separately. However we shall omit the diagram for $\Salg^1_4$ since this consists of the solitary orbit of $(9|3)$.

Before giving these diagrams we shall explain the notations that we use in these diagrams.

We represent the orbits of isomorphism classes of superalgebras, by using the $(i|j)$ notation from \cite{ACZ}; $(i|j)$ shall be used to denote the orbit $G_4 \cdot (i|j)$ in $\Salg_4$.

The families of superalgebras $(18;\l|i)$, $i=0,1,2$ consist of those superalgebras for all values of $\l$ except $-1$, which in particular includes the values $\l=0$ and $\l=1$. In these cases these orbits coincide with some of the other orbits. This is because, as superalgebras, we have the following equalities or isomorphisms: $(18;0|0)=(16|0),(18;0|1)=(16|1),(18;0|2)=(16|3),(18;1|0) \cong (7|0),(18;1|1) \cong (7|2),(18;1|2) \cong (7|3)$.

In the degeneration diagram we use a dashed line to indicate a ``degeneration" by a family of superalgebra structures; that is, when an orbit lies in the closure of the union of a family of orbits. This explains the use of the dashed lines through the families $(18;\l|i)$, $i=0,1,2$. The fact that we use an arrow from $(18;\l|0)$ to $(8|0)$ and from $(18;\l|2)$ to $(8|3)$ is because there is a genuine degeneration from each of the orbits
in these families to the orbits $(8|0)$ or $(8|3)$.

The dotted arrows (or dotted lines in the case of degenerations by a family of structures), are used to indicate those degenerations which we are unsure of --- there may or may not be a degeneration between the indicated superalgebras.

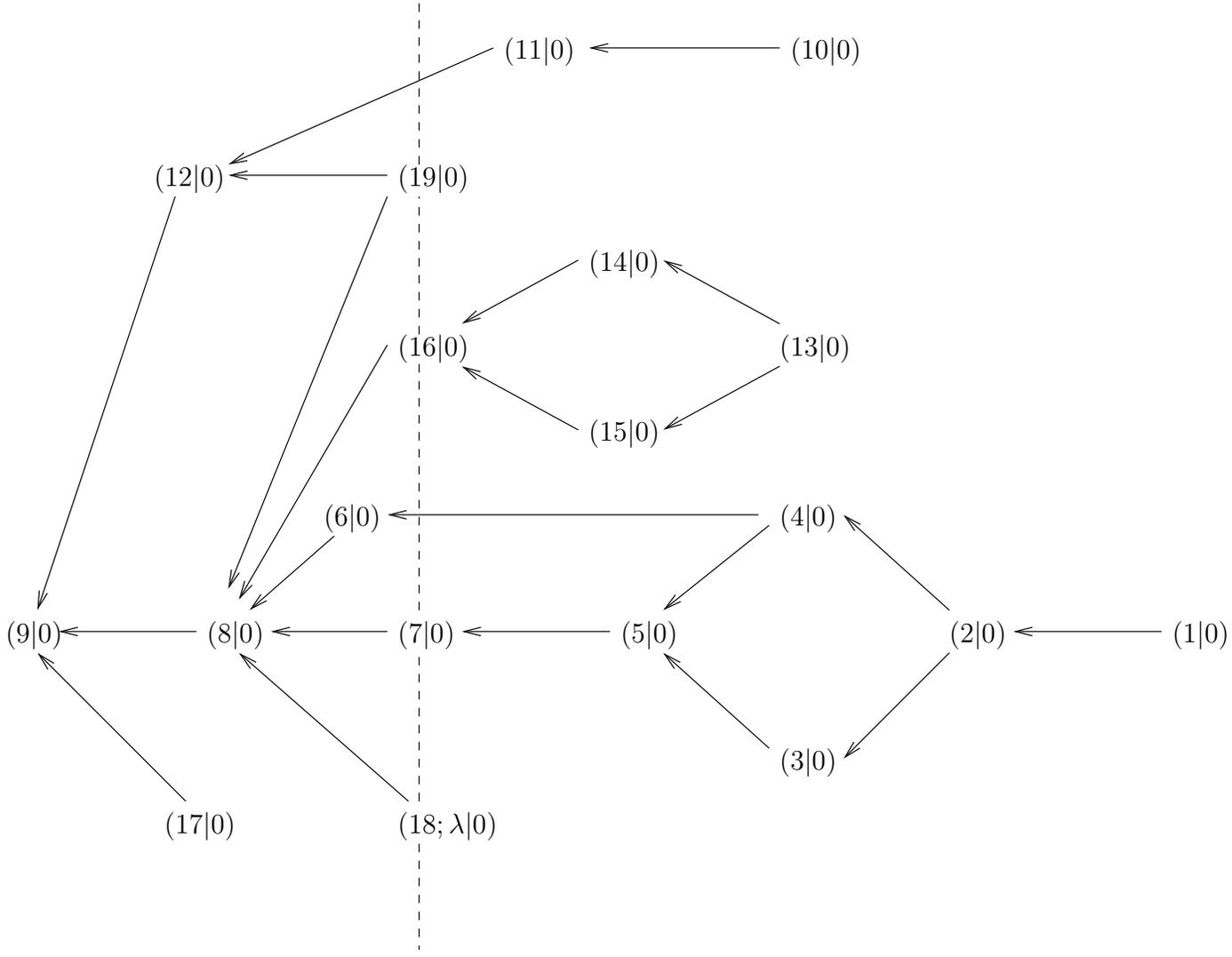
\begin{figure}[htbp]
\begin{center}
\rotatebox{90}{\scalebox{1.0}{\input{degendiag44.pstex_t}}}
\caption{Degenerations in the component $\Salg^4_4$}\label{dds44}
\end{center}
\end{figure}

\begin{figure}[htbp]
\begin{center}
\rotatebox{90}{\scalebox{0.5}{\input{degendiag43.pstex_t}}}
\caption{Degenerations in the component $\Salg^3_4$}\label{dds34}
\vspace{0.8cm}
\rotatebox{90}{\scalebox{0.5}{\input{degendiag42.pstex_t}}}
\caption{Degenerations in the component $\Salg^2_4$}\label{dds24}
\end{center}
\end{figure}

\newpage
From this we get the following (partial) result classifying $4$-dimensional superalgebras:

\begin{theorem}\thlabel{5.1} (Partial Geometric Classification of $4$-dimensional Superalgebras)\\
In $\Salg_4$ there are at least twenty irreducible components. The following structures (or families of structures) are known to be generic:

In $\Salg^4_4$: $(1|0),(10|0),(13|0),(17|0),(18;\l|0)$.

In $\Salg^3_4$: $(1|1),(11|1),(13|1),(14|1),(15|1),(17|1)$.

In $\Salg^2_4$: $(1|2),(10|1),(11|3),(14|3),(15|3),(17|2),(18;\l|1),(18;\l|2)$.

In $\Salg^1_4$: $(9|3)$.
\end{theorem}
\begin{proof}
This follows from the degeneration diagrams and \coref{3.12} which gives the relationship between the degeneration partial order and the irreducible components.
\end{proof}

\begin{remark}\relabel{5.2}
The result above guarantees the existence of twenty irreducible components, however there could be up to two more irreducible components as well. It is the connected component $\Salg^2_4$ in which we are unsure if we have found all of the irreducible components. It is not known whether the following two structures in $\Salg^2_4$ are generic or not: $(6|2),(19|1)$ --- so $\Salg^2_4$ could have as few as eight irreducible components or as many as ten.
\end{remark}

We are unsure if $(18;\l|2)$ degenerates to $(19|1)$ or not. This is why the dashed line through $(18;\l|2)$ changes to a dotted line after passing through $(16|3)$. We point this out to the reader to ensure that this important detail is not missed.

\begin{remark}\relabel{5.3}
\prref{3.17} gives the dimensions of these orbits, which for the generic structures gives the dimensions of the components too. However, for the generic families $(18;\l|i)$ for $i=0,1,2$, the dimension of the component must be at least one larger than the dimension of any single orbit in this family. Since the family depends on one parameter $\l$, we would suspect that the dimensions of these components of the generic families are exactly one larger than the dimension of any single orbit in this family. However, we have not proved this. To prove that this is indeed the case, it would suffice to show that there can be no closed irreducible set $Y$ lying properly between $\overline{G_n \cdot (18;\l|i)}$ and $\overline{\bigcup_{\l} G_n \cdot (18;\l|i)}$, i.e. that it is impossible to have $\overline{G_n \cdot (18;\l|i)} \subset Y \subset \overline{\bigcup_{\l} G_n \cdot (18;\l|i)}$ when $Y$ is closed and irreducible.
\end{remark}

We now provide the details which were used to obtain the degeneration diagrams just given:

We apply the following useful facts mentioned in \reref{4.7} in the previous section which shall help us here. Since $n=4$ we may appeal \leref{2.10} to see that $\Salg^i_4$ for $i=1,2,3,4$ are all closed disjoint subsets (and in fact by \prref{4.5} are the connected components of $\Salg_4$). Thus by part (c) of \reref{4.7} there cannot be a degeneration from $A$ to $B$ unless $\dim_0 A = \dim_0 B$. Thus we need only to look at the degenerations amongst superalgebras belonging to the same subset $\Salg^i_4$.

Another remark made in part (a) of \reref{4.7} is the following: If $U(A)$ doesn't degenerate to $U(B)$ as algebras, then $A$ cannot degenerate to $B$ as superalgebras. So we simply focus on degenerations from $A$ to $B$, when there is a degeneration from $U(A)$ to $U(B)$ of underlying algebras. These two remarks represent large simplifications for us, as they greatly reduce the number of degenerations we must consider. Since two different superalgebras on the same underlying algebra have a trivial degeneration of the underlying algebra, we must however check to see if there are degenerations between different superalgebras on the same underlying algebra.

We also recall, any superalgebra in $\Salg^i_4$ degenerates to the superalgebra structure on $k[X,Y,Z]/(X,Y,Z)^2$ in $\Salg^i_4$ for $i=1,2,3,4$. The orbit of this superalgebra is the closed orbit in $\Salg^i_4$. We will not mention this degeneration further since it always exists. We gave the specialization giving rise to this degeneration in \reref{4.4}.

By \coref{4.3}, to show the existence of a degeneration, it suffices to exhibit a specialization. In this section to show the existence of degenerations we shall do this, except in one instance where we shall appeal to \leref{4.1} directly.

We mention that all the specializations given in this section are ``homogeneous", that is, the basis changes replace degree zero terms by degree zero terms, and similarly replace degree one terms by degree one terms. \coref{4.3} applies equally well to non-homogeneous specializations, however, such specializations are more difficult to determine. In fact, there are some superalgebras which we haven't determined whether there is or is not a degeneration between (e.g. does $(1|2)$ degenerate to $(6|2)$?), but if the degeneration was to be obtained by a specialization it would necessarily have to be non-homogeneous. For an example of a degeneration obtained by a non-homogeneous specialization we have the following in the dimension 2 case, where each superalgebra is given the non-trivial $\mathbb{Z}_2$-grading:

$k \times k \rightarrow k[X]/(X^2)$ by $e_1=(1,1),e_2=(1,-1),e'_1=e_1,e'_2=te_1+te_2$ let $t \rightarrow 0$

To show the non-existence of a degeneration we list the method which we use. There are several different methods. We give the name and a brief explanation for each below.

\begin{itemize}
\item{By \leref{3.6} part (c) the orbit dimension must strictly decrease upon proper degeneration. So a superalgebra cannot degenerate to another superalgebra of the same or greater dimension. We abbreviate this method by (OD). Note however that it is possible for a family of structures of a given dimension to ``degenerate" to a structure of the same dimension. As an example of this, each orbit in $(18;\l|0)$ has dimension $8$ as does the orbit $(19|0)$, yet the family $(18;\l|0)$ ``degenerates" to $(19|0)$.}
\item{For the other methods we use the closed $G_n$-stable subsets found in the previous section. If $A$ belongs to one of these subsets, and $B$ does not, then $A$ cannot degenerate to $B$. We shall refer to this set of methods by which of the closed $G_n$-stable subsets we apply. The abbreviation we give to the method by applying one of the closed sets is listed below.

\begin{itemize}
\item{(A) $\{A \in \Salg_n:A_1^2=\{0\}\}$}.
\item{(B) $\{A \in \Salg_n:A_0$ is commutative $\}$}.
\item{(C) $\{A\in \Salg^2_4: \dim J(A_0)=1\}$}.
\item{(D) $\{A\in \Salg^2_4: \dim J(A_0)=1,J(A_0)A_1=\{0\}\}$}.
\item{(E) $\{A\in \Salg^2_4: \dim J(A_0)=1,A_1J(A_0)=\{0\}\}$}.
\end{itemize}}
\end{itemize}

In the following, when $\a \neq 0$, we will use the shorthand, $\sqrt{\a}$ to denote {\em some element}, $x$, of $k^{*}$, such that $x^2=\a$. (Such an element $x$ always exists as $k$ is algebraically closed. Moreover, if $x$ is such an element, then so too is $-x$).

\bigskip

{\Large Case $\dim_0=4$:}

\medskip

Applying part (b) in \reref{4.7} from the previous section, we notice that the degeneration diagram of $\Salg^4_4$ corresponds exactly to the degeneration diagram of $\Alg_4$. These degenerations have been completely described by Gabriel in \cite{Gab}, where he gives the degeneration diagram.

\medskip

{\Large Case $\dim_0=3$:}

\medskip

Existence of Degenerations:

\smallskip

$(1|1)\rightarrow (2|1):$ $e_1=(1,1,1,1),e_2=(0,1,0,0),e_3=(0,0,1,1),e_4=(0,0,1,-1),e'_1=e_1,e'_2=e_2,e'_3=e_3,e'_4=te_4$ let $t \rightarrow 0$.

$(1|1)\rightarrow (2|2):$ $e_1=(1,1,1,1),e_2=(0,0,1,1),e_3=(1,-1,0,0),e_4=(0,0,1,-1),e'_1=e_1,e'_2=e_2,e'_3=te_3,e'_4=e_4$ let $t \rightarrow 0$.

$(1|1)\rightarrow (4|1):$ $e_1=(1,1,1,1),e_2=(1,0,0,0),e_3=(0,0,1,1),e_4=(0,0,1,-1),e'_1=e_1,e'_2=e_2,e'_3=t^2e_3,e'_4=te_4$ let $t \rightarrow 0$.

$(2|1)\rightarrow (3|1):$ $e_1=(1,1,1),e_2=(1,1,0),e_3=(1,-1,0),e_4=(0,0,X),e'_1=e_1,e'_2=e_2,e'_3=te_3,e'_4=e_4$ let $t \rightarrow 0$.

$(2|1)\rightarrow (6|1):$ $e_1=(1,1,1),e_2=(1,0,0),e_3=(0,-1,1),e_4=(0,0,X),e'_1=e_1,e'_2=e_2,e'_3=te_3,e'_4=e_4$ let $t \rightarrow 0$

$(2|2)\rightarrow (3|1):$ $e_1=(1,1,1),e_2=(1,1,0),e_3=(0,0,X),e_4=(1,-1,0),e'_1=e_1,e'_2=e_2,e'_3=e_3,e'_4=te_4$ let $t \rightarrow 0$.

$(2|2)\rightarrow (7|1):$ $e_1=(1,1,1),e_2=(1,1,0),e_3=(0,0,X),e_4=(1,-1,0),e'_1=e_1,e'_2=\sqrt{2}te_2+e_3,e'_3=t^2e_2,e'_4=\sqrt{-2}te_4$ let $t \rightarrow 0$.

$(3|1)\rightarrow (8|1):$ $e_1=(1,1),e_2=(1,0),e_3=(X,0),e_4=(0,Y),e'_1=e_1,e'_2=te_2+e_3,e'_3=te_3,e'_4=e_4$ let $t \rightarrow 0$.

$(4|1)\rightarrow (6|1):$ $e_1=(1,1),e_2=(1,0),e_3=(0,X^2),e_4=(0,X),e'_1=e_1,e'_2=e_2,e'_3=e_3,e'_4=te_4$ let $t \rightarrow 0$.

$(4|1)\rightarrow (7|1):$ $e_1=(1,1),e_2=(-1,1),e_3=(0,X^2),e_4=(0,X),e'_1=e_1,e'_2=t^2e_2+e_3,e'_3=t^2e_3,e'_4=\sqrt{-2}te_4$ let $t \rightarrow 0$.

$(6|1)\rightarrow (8|1):$ $e_1=(1,1),e_2=(-1,1),e_3=(0,X),e_4=(0,Y),e'_1=e_1,e'_2=te_2+e_3,e'_3=2te_3,e'_4=e_4$ let $t \rightarrow 0$.

$(7|1)\rightarrow (8|1):$ $e_1=1,e_2=X+Y,e_3=XY,e_4=X-Y,e'_1=e_1,e'_2=e_2,e'_3=2e_3,e'_4=te_4$ let $t \rightarrow 0$.

$(7|1)\rightarrow (8|2):$ $e_1=1,e_2=X+Y,e_3=XY,e_4=X-Y,e'_1=e_1,e'_2=-2e_3,e'_3=te_2,e'_4=e_4$ let $t \rightarrow 0$.

$(13|1)\rightarrow (14|2):$ $e_1=\left(1,{\scriptsize\begin{pmatrix} 1 & 0 \\ 0 & 1\\ \end{pmatrix}}\right),e_2=\left(1,{\scriptsize\begin{pmatrix} 0 & 0 \\ 0 & 1\\ \end{pmatrix}}\right),e_3=\left(1,{\scriptsize\begin{pmatrix} 0 & 0 \\ 0 & 0\\ \end{pmatrix}}\right),e_4=\left(0,{\scriptsize\begin{pmatrix} 0 & 1 \\ 0 & 0\\ \end{pmatrix}}\right),e'_1=e_1,e'_2=e_2,e'_3=te_3,e'_4=e_4$ let $t \rightarrow 0$.

$(13|1)\rightarrow (15|2):$ $e_1=\left(1,{\scriptsize\begin{pmatrix} 1 & 0 \\ 0 & 1\\ \end{pmatrix}}\right),e_2=\left(1,{\scriptsize\begin{pmatrix} 1 & 0 \\ 0 & 0\\ \end{pmatrix}}\right),e_3=\left(1,{\scriptsize\begin{pmatrix} 0 & 0 \\ 0 & 0\\ \end{pmatrix}}\right),e_4=\left(0,{\scriptsize\begin{pmatrix} 0 & 1 \\ 0 & 0\\ \end{pmatrix}}\right),e'_1=e_1,e'_2=e_2,e'_3=te_3,e'_4=e_4$ let $t \rightarrow 0$.

$(14|2)\rightarrow (8|1):$ $e_1={\scriptsize\begin{pmatrix} 1 & 0 & 0\\ 0 & 1 & 0\\ 0 & 0 & 1\\ \end{pmatrix}},e_2={\scriptsize\begin{pmatrix} 1 & 0 & 0\\ 0 & 1 & 0\\ 0 & 0 & -1\\ \end{pmatrix}},e_3={\scriptsize\begin{pmatrix} 0 & 0 & 0\\ 1 & 0 & 0\\ 0 & 0 & 0\\ \end{pmatrix}},e_4={\scriptsize\begin{pmatrix} 0 & 0 & 0\\ 0 & 0 & 0\\ 1 & 0 & 0\\ \end{pmatrix}} ,e'_1=e_1,e'_2=te_2+e_3,e'_3=2te_3,e'_4=e_4$ let $t \rightarrow 0$.

$(15|2)\rightarrow (8|1):$ $e_1={\scriptsize\begin{pmatrix} 1 & 0 & 0\\ 0 & 1 & 0\\ 0 & 0 & 1\\ \end{pmatrix}},e_2={\scriptsize\begin{pmatrix} 1 & 0 & 0\\ 0 & 1 & 0\\ 0 & 0 & -1\\ \end{pmatrix}},e_3= {\scriptsize\begin{pmatrix} 0 & 1 & 0\\ 0 & 0 & 0\\ 0 & 0 & 0\\ \end{pmatrix}},e_4={\scriptsize\begin{pmatrix} 0 & 0 & 1\\ 0 & 0 & 0\\ 0 & 0 & 0\\ \end{pmatrix}},e'_1=e_1,e'_2=te_2+e_3,e'_3=2te_3,e'_4=e_4$ let $t \rightarrow 0$.

\medskip

Non-existence of Degenerations:

\smallskip

$(2|1) \nrightarrow (2|2)$ (OD);

$(2|1) \nrightarrow (4|1)$ (OD);

$(2|1) \nrightarrow (7|1)$ (A);

$(2|1) \nrightarrow (8|2)$ (A);

$(2|2) \nrightarrow (2|1)$ (OD);

$(2|2) \nrightarrow (4|1)$ (OD);

$(3|1) \nrightarrow (7|1)$ (OD);

$(3|1) \nrightarrow (8|2)$ (A);

$(6|1) \nrightarrow (8|2)$ (A);

$(8|1) \nrightarrow (8|2)$ (OD);

$(8|2) \nrightarrow (8|1)$ (OD);

$(13|1) \nrightarrow (8|2)$ (A);

$(13|1) \nrightarrow (14|1)$ (B);

$(13|1) \nrightarrow (15|1)$ (B);

$(14|1) \nrightarrow (8|1)$ (OD);

$(14|1) \nrightarrow (8|2)$ (OD);

$(14|1) \nrightarrow (14|2)$ (OD);

$(14|2) \nrightarrow (8|2)$ (A);

$(14|2) \nrightarrow (14|1)$ (B);

$(15|1) \nrightarrow (8|1)$ (OD);

$(15|1) \nrightarrow (8|2)$ (OD);

$(15|1) \nrightarrow (15|2)$ (OD);

$(15|2) \nrightarrow (8|2)$ (A);

$(15|2) \nrightarrow (15|1)$ (B).

\medskip

Undetermined Degeneration:

\smallskip

$(2|2) \qmra (6|1)$.
\medskip

{\Large Case $\dim_0=2$:}

\medskip

Existence of Degenerations:

\smallskip

$(1|2)\rightarrow (2|3):$ $e_1=(1,1,1,1),e_2=(0,0,1,1),e_3=(1,-1,0,0),e_4=(0,0,1,-1),e'_1=e_1,e'_2=e_2,e'_3=e_3,e'_4=te_4$ let $t \rightarrow 0$.

$(1|2)\rightarrow (3|3):$ $e_1=(1,1,1,1),e_2=(1,1,0,0),e_3=(1,-1,1,-1),e_4=(1,-1,0,0),e'_1=e_1,e'_2=te_2,e'_3=e_3,e'_4=te_4$ let $t \rightarrow 0$.

$(2|3)\rightarrow (3|2):$ $e_1=(1,1,1),e_2=(1,1,0),e_3=(1,-1,0),e_4=(0,0,X),e'_1=e_1,e'_2=e_2,e'_3=te_3,e'_4=e_4$ let $t \rightarrow 0$.

$(2|3)\rightarrow (5|1):$ $e_1=(1,1,1),e_2=(1,1,0),e_3=(1,-1,0),e_4=(0,0,X),e'_1=e_1,e'_2=t^2e_2,e'_3=te_3+e_4,e'_4=t^3e_3$ let $t \rightarrow 0$.

$(3|2)\rightarrow (7|2):$ $e_1=(1,1),e_2=(1,-1),e_3=(X,Y),e_4=(X,-Y),e'_1=e_1,e'_2=te_2,e'_3=e_3,e'_4=te_4$ let $t \rightarrow 0$.

$(3|3)\rightarrow (5|1):$ $e_1=(1,1),e_2=(X,Y),e_3=(1,-1),e_4=(X,-Y),e'_1=e_1,e'_2=2te_2,e'_3=te_3+e_4,e'_4=2t^2e_4$ let $t \rightarrow 0$.

$(3|3)\rightarrow (7|3):$ $e_1=(1,1),e_2=(X,Y),e_3=(1,-1),e_4=(X,-Y),e'_1=e_1,e'_2=te_2,e'_3=te_3,e'_4=e_4$ let $t \rightarrow 0$.

$(5|1)\rightarrow (7|2):$ $e_1=1,e_2=X^2,e_3=X,e_4=X^3,e'_1=e_1,e'_2=e_2,e'_3=te_3,e'_4=te_4$ let $t \rightarrow 0$.

$(5|1)\rightarrow (8|3):$ $e_1=1,e_2=X^2,e_3=X,e_4=X^3,e'_1=e_1,e'_2=t^2e_2,e'_3=te_3,e'_4=e_4$ let $t \rightarrow 0$

$(7|3)\rightarrow (8|3):$ $e_1=1,e_2=XY,e_3=X+Y,e_4=X-Y,e'_1=e_1,e'_2=2e_2,e'_3=e_3,e'_4=te_4$ let $t \rightarrow 0$.

$(10|1)\rightarrow (11|2):$ $e_1={\scriptsize\begin{pmatrix} 1 & 0\\ 0 & 1\\ \end{pmatrix}},e_2={\scriptsize\begin{pmatrix} 1 & 0\\ 0 & -1\\ \end{pmatrix}},e_3={\scriptsize\begin{pmatrix} 0 & 1\\ 0 & 0\\ \end{pmatrix}},e_4={\scriptsize\begin{pmatrix} 0 & 0\\ 1 & 0\\ \end{pmatrix}},e'_1=e_1,e'_2=e_2,e'_3=te_3,e'_4=te_4$ let $t \rightarrow 0$.

$(10|1)\rightarrow (12|2):$ $e_1={\scriptsize\begin{pmatrix} 1 & 0\\ 0 & 1\\ \end{pmatrix}},e_2={\scriptsize\begin{pmatrix} 1 & 0\\ 0 & -1\\ \end{pmatrix}},e_3={\scriptsize\begin{pmatrix} 0 & 1\\ -1 & 0\\ \end{pmatrix}},e_4={\scriptsize\begin{pmatrix} 0 & 1\\ 1 & 0\\ \end{pmatrix}},e'_1=e_1,e'_2=t^2e_2,e'_3=te_3,e'_4=te_4$ let $t \rightarrow 0$.

$(11|2)\rightarrow (12|1):$ $e_1={\scriptsize\begin{pmatrix} 1 & 0 & 0 & 0\\ 0 & 1 & 0 & 0\\ 0 & 0 & 1 & 0\\ 0 & 0 & 0 & 1\\ \end{pmatrix}},e_2={\scriptsize\begin{pmatrix} 1 & 0 & 0 & 0\\ 0 & 1 & 0 & 0\\ 0 & 0 & -1 & 0\\ 0 & 0 & 0 & -1\\ \end{pmatrix}},e_3={\scriptsize\begin{pmatrix} 0 & 0 & 0 & 0\\ 0 & 0 & 0 & 1\\ 1 & 0 & 0 & 0\\ 0 & 0 & 0 & 0\\ \end{pmatrix}}, \linebreak e_4={\scriptsize\begin{pmatrix} 0 & 0 & 0 & 0\\ 0 & 0 & 0 & 1\\ -1 & 0 & 0 & 0\\ 0 & 0 & 0 & 0\\ \end{pmatrix}},e'_1=e_1,e'_2=te_2,e'_3=e_3,e'_4=te_4$ let $t \rightarrow 0$.

$(11|3)\rightarrow (12|1):$ $e_1={\scriptsize\begin{pmatrix} 1 & 0 & 0 & 0\\ 0 & 1 & 0 & 0\\ 0 & 0 & 1 & 0\\ 0 & 0 & 0 & 1\\ \end{pmatrix}},e_2={\scriptsize\begin{pmatrix} 0 & 0 & 0 & 0\\ 0 & 0 & 0 & 1\\ 1 & 0 & 0 & 0\\ 0 & 0 & 0 & 0\\ \end{pmatrix}},e_3={\scriptsize\begin{pmatrix} 1 & 0 & 0 & 0\\ 0 & 1 & 0 & 0\\ 0 & 0 & -1 & 0\\ 0 & 0 & 0 & -1\\ \end{pmatrix}}, \linebreak e_4={\scriptsize\begin{pmatrix} 0 & 0 & 0 & 0\\ 0 & 0 & 0 & -1\\ 1 & 0 & 0 & 0\\ 0 & 0 & 0 & 0\\ \end{pmatrix}},e'_1=e_1,e'_2=e_2,e'_3=te_3,e'_4=te_4$ let $t \rightarrow 0$.

$(11|3)\rightarrow (12|2):$ $e_1={\scriptsize\begin{pmatrix} 1 & 0 & 0 & 0\\ 0 & 1 & 0 & 0\\ 0 & 0 & 1 & 0\\ 0 & 0 & 0 & 1\\ \end{pmatrix}},e_2={\scriptsize\begin{pmatrix} 0 & 0 & 0 & 0\\ 0 & 0 & 0 & 1\\ 1 & 0 & 0 & 0\\ 0 & 0 & 0 & 0\\ \end{pmatrix}},e_3={\scriptsize\begin{pmatrix} 1 & 0 & 0 & 0\\ 0 & 1 & 0 & 0\\ 0 & 0 & -1 & 0\\ 0 & 0 & 0 & -1\\ \end{pmatrix}}, \linebreak e_4={\scriptsize\begin{pmatrix} 0 & 0 & 0 & 0\\ 0 & 0 & 0 & 1\\ -1 & 0 & 0 & 0\\ 0 & 0 & 0 & 0\\ \end{pmatrix}},e'_1=e_1,e'_2=te_2,e'_3=te_3,e'_4=e_4$ let $t \rightarrow 0$.

$(14|3)\rightarrow (16|1):$ $e_1={\scriptsize\begin{pmatrix} 1 & 0 & 0\\ 0 & 1 & 0\\ 0 & 0 & 1\\ \end{pmatrix}},e_2={\scriptsize\begin{pmatrix} 1 & 0 & 0\\ 0 & 1 & 0\\ 0 & 0 & -1\\ \end{pmatrix}},e_3={\scriptsize\begin{pmatrix} 0 & 0 & 0\\ 1 & 0 & 0\\ 1 & 0 & 0\\ \end{pmatrix}},e_4={\scriptsize\begin{pmatrix} 0 & 0 & 0\\ 1 & 0 & 0\\ -1 & 0 & 0\\ \end{pmatrix}},e'_1=e_1,e'_2=te_2,e'_3=e_3,e'_4=te_4$ let $t \rightarrow 0$.

$(15|3)\rightarrow (16|2):$ $e_1={\scriptsize\begin{pmatrix} 1 & 0 & 0\\ 0 & 1 & 0\\ 0 & 0 & 1\\ \end{pmatrix}},e_2={\scriptsize\begin{pmatrix} 1 & 0 & 0\\ 0 & 1 & 0\\ 0 & 0 & -1\\ \end{pmatrix}},e_3={\scriptsize\begin{pmatrix} 0 & 1 & 1\\ 0 & 0 & 0\\ 0 & 0 & 0\\ \end{pmatrix}},e_4={\scriptsize\begin{pmatrix} 0 & 1 & -1\\ 0 & 0 & 0\\ 0 & 0 & 0\\ \end{pmatrix}},e'_1=e_1,e'_2=te_2,e'_3=e_3,e'_4=te_4$ let $t \rightarrow 0$.

$(16|3)\rightarrow (8|3):$ $e_1=1,e_2=XY,e_3=X+Y,e_4=X-Y,e'_1=e_1,e'_2=e_2,e'_3=e_3,e'_4=te_4$ let $t \rightarrow 0$.

$(18;\l|1) \rightarrow (7|2),(16|1):$
Since the orbits of $(7|2)$ and $(16|1)$ coincide with the orbits of $(18;1|1)$ and $(18;0|1)$ respectively, $(7|2)$ and $(16|1)$ are included in the closure of the union of the family of orbits $(18;\l|1)$.

$(18;\l|1) \rightarrow (16|2):$
Also $(16|2)$ is included in the closure of the union of the family of orbits $(18;\l|1)$. To see this, we look at the structure constants of $(18;t^{-1}|1)$ in the basis $e_1=1,e_2=X,e_3=Y,e_4=YX$. This gives us a curve in $\Salg_4$ which lies in the family of orbits of $(18;\l|1)$ for $t \neq 0$, yet lies in the orbit of $(16|2)$ when $t=0$. By appealing to \leref{4.1} directly the result follows.

$(18;\l|2) \rightarrow (7|3),(16|3):$
Similarly the orbits of $(7|3)$ and $(16|3)$ are included in the closure of the union of the family of orbits $(18;\l|2)$.

$(18;\l|2)\rightarrow (8|3):$ $e_1=1,e_2=XY,e_3=X+Y,e_4=X-Y,e'_1=e_1,e'_2=(1+\l)e_2,e'_3=e_3,e'_4=te_4$ let $t \rightarrow 0$.

$(19|1)\rightarrow (8|3):$ $e_1=1,e_2=XY,e_3=X+Y,e_4=X-Y,e'_1=e_1,e'_2=e_2,e'_3=e_3,e'_4=te_4$ let $t \rightarrow 0$.

$(19|1)\rightarrow (12|2):$ $e_1=1,e_2=XY,e_3=X,e_4=Y,e'_1=e_1,e'_2=te_2,e'_3=te_3,e'_4=e_4$ let $t \rightarrow 0$.

\medskip

Non-existence of Degenerations:

\smallskip

$(2|3) \nrightarrow (3|3)$ (OD);

$(3|2) \nrightarrow (3|3)$ (OD);

$(3|2) \nrightarrow (5|1)$ (OD);

$(3|2) \nrightarrow (7|3)$ (OD);

$(3|2) \nrightarrow (8|3)$ (A);

$(3|3) \nrightarrow (3|2)$ (C);

$(5|1) \nrightarrow (7|3)$ (OD);

$(6|2) \nrightarrow (8|3)$ (OD);

$(7|2) \nrightarrow (7|3)$ (OD);

$(7|2) \nrightarrow (8|3)$ (OD);

$(7|3) \nrightarrow (7|2)$ (D);

$(10|1) \nrightarrow (11|3)$ (OD);

$(11|2) \nrightarrow (11|3)$ (OD);

$(11|2) \nrightarrow (12|2)$ (A);

$(11|3) \nrightarrow (11|2)$ (C);

$(12|1) \nrightarrow (12|2)$ (A);

$(12|2) \nrightarrow (12|1)$ (OD);

$(14|3) \nrightarrow (16|3)$ (OD);

$(14|3) \nrightarrow (8|3)$ (A);

$(15|3) \nrightarrow (16|3)$ (OD);

$(15|3) \nrightarrow (8|3)$ (A);

$(16|1) \nrightarrow (16|2)$ (OD);

$(16|1) \nrightarrow (16|3)$ (OD);

$(16|1) \nrightarrow (8|3)$ (OD);

$(16|2) \nrightarrow (16|1)$ (OD);

$(16|2) \nrightarrow (16|3)$ (OD);

$(16|2) \nrightarrow (8|3)$ (OD);

$(16|3) \nrightarrow (16|1)$ (D);

$(16|3) \nrightarrow (16|2)$ (E);

$(18;\l|1) \nrightarrow (7|3),(16|3),(18;\l|2),(19|1)$ (A);

$(18;\l|1) \nrightarrow (8|3)$ (A);

$(18;\l|2) \nrightarrow (7|2),(16|1),(16|2),(18;\l|1)$ (D), (E);

$(19|1) \nrightarrow (12|1)$ (D).

\medskip

Undetermined Degenerations:

\smallskip

$(1|2) \qmra (6|2)$;

$(2|3) \qmra (6|2)$;

$(18;\l|2) \qmra (19|1)$;

$(2|3) \qmra (7|3)$;

$(14|3) \qmra (16|2)$;

$(15|3) \qmra (16|1)$.

The first three of these undetermined degenerations are related to discovering whether $(6|2)$ or $(19|1)$ give rise to irreducible components in $\Salg^2_4$.

\begin{remark}\relabel{5.4}
We close with the remark that in $\Salg_4$ no two superalgebra structures $A$ and $B$ on the same underlying algebra can degenerate to each other, even if $\dim_0 A=\dim_0 B$. We have seen this from brute force checking of each case. Is it a general result that there can be no degeneration from a superalgebra to any other superalgebra having the same underlying algebra?
\end{remark}

\end{document}

%% file: degendiag44.pstex_t
\begin{picture}(0,0)%
\includegraphics{degendiag44.pstex}%
\end{picture}%
\setlength{\unitlength}{3947sp}%
\begingroup\makeatletter\ifx\SetFigFont\undefined%
\gdef\SetFigFont#1#2#3#4#5{%
  \reset@font\fontsize{#1}{#2pt}%
  \fontfamily{#3}\fontseries{#4}\fontshape{#5}%
  \selectfont}%
\fi\endgroup%
\begin{picture}(8819,6774)(376,-8323)
\put(4501,-3511){\makebox(0,0)[lb]{\smash{{\SetFigFont{12}{14.4}{\familydefault}{\mddefault}{\updefault}{\color[rgb]{0,0,0}$(14|0)$}%
}}}}
\put(376,-6136){\makebox(0,0)[lb]{\smash{{\SetFigFont{12}{14.4}{\familydefault}{\mddefault}{\updefault}{\color[rgb]{0,0,0}$(9|0)$}%
}}}}
\put(1501,-7486){\makebox(0,0)[lb]{\smash{{\SetFigFont{12}{14.4}{\familydefault}{\mddefault}{\updefault}{\color[rgb]{0,0,0}$(17|0)$}%
}}}}
\put(1426,-2911){\makebox(0,0)[lb]{\smash{{\SetFigFont{12}{14.4}{\familydefault}{\mddefault}{\updefault}{\color[rgb]{0,0,0}$(12|0)$}%
}}}}
\put(3151,-4111){\makebox(0,0)[lb]{\smash{{\SetFigFont{12}{14.4}{\familydefault}{\mddefault}{\updefault}{\color[rgb]{0,0,0}$(16|0)$}%
}}}}
\put(3151,-2911){\makebox(0,0)[lb]{\smash{{\SetFigFont{12}{14.4}{\familydefault}{\mddefault}{\updefault}{\color[rgb]{0,0,0}$(19|0)$}%
}}}}
\put(4501,-4711){\makebox(0,0)[lb]{\smash{{\SetFigFont{12}{14.4}{\familydefault}{\mddefault}{\updefault}{\color[rgb]{0,0,0}$(15|0)$}%
}}}}
\put(5851,-4111){\makebox(0,0)[lb]{\smash{{\SetFigFont{12}{14.4}{\familydefault}{\mddefault}{\updefault}{\color[rgb]{0,0,0}$(13|0)$}%
}}}}
\put(2626,-5311){\makebox(0,0)[lb]{\smash{{\SetFigFont{12}{14.4}{\familydefault}{\mddefault}{\updefault}{\color[rgb]{0,0,0}$(6|0)$}%
}}}}
\put(1801,-6136){\makebox(0,0)[lb]{\smash{{\SetFigFont{12}{14.4}{\familydefault}{\mddefault}{\updefault}{\color[rgb]{0,0,0}$(8|0)$}%
}}}}
\put(3151,-7486){\makebox(0,0)[lb]{\smash{{\SetFigFont{12}{14.4}{\familydefault}{\mddefault}{\updefault}{\color[rgb]{0,0,0}$(18;\l|0)$}%
}}}}
\put(3151,-6136){\makebox(0,0)[lb]{\smash{{\SetFigFont{12}{14.4}{\familydefault}{\mddefault}{\updefault}{\color[rgb]{0,0,0}$(7|0)$}%
}}}}
\put(4726,-6136){\makebox(0,0)[lb]{\smash{{\SetFigFont{12}{14.4}{\familydefault}{\mddefault}{\updefault}{\color[rgb]{0,0,0}$(5|0)$}%
}}}}
\put(5851,-5311){\makebox(0,0)[lb]{\smash{{\SetFigFont{12}{14.4}{\familydefault}{\mddefault}{\updefault}{\color[rgb]{0,0,0}$(4|0)$}%
}}}}
\put(5851,-7036){\makebox(0,0)[lb]{\smash{{\SetFigFont{12}{14.4}{\familydefault}{\mddefault}{\updefault}{\color[rgb]{0,0,0}$(3|0)$}%
}}}}
\put(7051,-6136){\makebox(0,0)[lb]{\smash{{\SetFigFont{12}{14.4}{\familydefault}{\mddefault}{\updefault}{\color[rgb]{0,0,0}$(2|0)$}%
}}}}
\put(8626,-6136){\makebox(0,0)[lb]{\smash{{\SetFigFont{12}{14.4}{\familydefault}{\mddefault}{\updefault}{\color[rgb]{0,0,0}$(1|0)$}%
}}}}
\put(3901,-2011){\makebox(0,0)[lb]{\smash{{\SetFigFont{12}{14.4}{\familydefault}{\mddefault}{\updefault}{\color[rgb]{0,0,0}$(11|0)$}%
}}}}
\put(5926,-2011){\makebox(0,0)[lb]{\smash{{\SetFigFont{12}{14.4}{\familydefault}{\mddefault}{\updefault}{\color[rgb]{0,0,0}$(10|0)$}%
}}}}
\end{picture}%

%% file: degendiag43.pstex_t
\begin{picture}(0,0)%
\includegraphics{degendiag43.pstex}%
\end{picture}%
\setlength{\unitlength}{3947sp}%
\begingroup\makeatletter\ifx\SetFigFont\undefined%
\gdef\SetFigFont#1#2#3#4#5{%
  \reset@font\fontsize{#1}{#2pt}%
  \fontfamily{#3}\fontseries{#4}\fontshape{#5}%
  \selectfont}%
\fi\endgroup%
\begin{picture}(8819,5689)(376,-7544)
\put(5851,-4111){\makebox(0,0)[lb]{\smash{{\SetFigFont{12}{14.4}{\familydefault}{\mddefault}{\updefault}{\color[rgb]{0,0,0}$(13|1)$}%
}}}}
\put(8626,-6136){\makebox(0,0)[lb]{\smash{{\SetFigFont{12}{14.4}{\familydefault}{\mddefault}{\updefault}{\color[rgb]{0,0,0}$(1|1)$}%
}}}}
\put(5851,-5311){\makebox(0,0)[lb]{\smash{{\SetFigFont{12}{14.4}{\familydefault}{\mddefault}{\updefault}{\color[rgb]{0,0,0}$(4|1)$}%
}}}}
\put(5851,-7036){\makebox(0,0)[lb]{\smash{{\SetFigFont{12}{14.4}{\familydefault}{\mddefault}{\updefault}{\color[rgb]{0,0,0}$(3|1)$}%
}}}}
\put(3151,-6136){\makebox(0,0)[lb]{\smash{{\SetFigFont{12}{14.4}{\familydefault}{\mddefault}{\updefault}{\color[rgb]{0,0,0}$(7|1)$}%
}}}}
\put(3901,-2011){\makebox(0,0)[lb]{\smash{{\SetFigFont{12}{14.4}{\familydefault}{\mddefault}{\updefault}{\color[rgb]{0,0,0}$(11|1)$}%
}}}}
\put(376,-6136){\makebox(0,0)[lb]{\smash{{\SetFigFont{12}{14.4}{\familydefault}{\mddefault}{\updefault}{\color[rgb]{0,0,0}$(9|1)$}%
}}}}
\put(1501,-7486){\makebox(0,0)[lb]{\smash{{\SetFigFont{12}{14.4}{\familydefault}{\mddefault}{\updefault}{\color[rgb]{0,0,0}$(17|1)$}%
}}}}
\put(4501,-4261){\makebox(0,0)[lb]{\smash{{\SetFigFont{12}{14.4}{\familydefault}{\mddefault}{\updefault}{\color[rgb]{0,0,0}$(15|1)$}%
}}}}
\put(4501,-3286){\makebox(0,0)[lb]{\smash{{\SetFigFont{12}{14.4}{\familydefault}{\mddefault}{\updefault}{\color[rgb]{0,0,0}$(14|1)$}%
}}}}
\put(4501,-3661){\makebox(0,0)[lb]{\smash{{\SetFigFont{12}{14.4}{\familydefault}{\mddefault}{\updefault}{\color[rgb]{0,0,0}$(14|2)$}%
}}}}
\put(4501,-4636){\makebox(0,0)[lb]{\smash{{\SetFigFont{12}{14.4}{\familydefault}{\mddefault}{\updefault}{\color[rgb]{0,0,0}$(15|2)$}%
}}}}
\put(1801,-5911){\makebox(0,0)[lb]{\smash{{\SetFigFont{12}{14.4}{\familydefault}{\mddefault}{\updefault}{\color[rgb]{0,0,0}$(8|1)$}%
}}}}
\put(1801,-6436){\makebox(0,0)[lb]{\smash{{\SetFigFont{12}{14.4}{\familydefault}{\mddefault}{\updefault}{\color[rgb]{0,0,0}$(8|2)$}%
}}}}
\put(7126,-5761){\makebox(0,0)[lb]{\smash{{\SetFigFont{12}{14.4}{\rmdefault}{\mddefault}{\updefault}{\color[rgb]{0,0,0}$(2|1)$}%
}}}}
\put(7126,-6511){\makebox(0,0)[lb]{\smash{{\SetFigFont{12}{14.4}{\rmdefault}{\mddefault}{\updefault}{\color[rgb]{0,0,0}$(2|2)$}%
}}}}
\put(2626,-5311){\makebox(0,0)[lb]{\smash{{\SetFigFont{12}{14.4}{\familydefault}{\mddefault}{\updefault}{\color[rgb]{0,0,0}$(6|1)$}%
}}}}
\end{picture}%

%% file: degendiag42.pstex_t
\begin{picture}(0,0)%
\includegraphics{degendiag42.pstex}%
\end{picture}%
\setlength{\unitlength}{3947sp}%
\begingroup\makeatletter\ifx\SetFigFont\undefined%
\gdef\SetFigFont#1#2#3#4#5{%
  \reset@font\fontsize{#1}{#2pt}%
  \fontfamily{#3}\fontseries{#4}\fontshape{#5}%
  \selectfont}%
\fi\endgroup%
\begin{picture}(9119,6774)(376,-8173)
\put(2776,-3736){\makebox(0,0)[lb]{\smash{{\SetFigFont{12}{14.4}{\familydefault}{\mddefault}{\updefault}{\color[rgb]{0,0,0}$(16|2)$}%
}}}}
\put(376,-6136){\makebox(0,0)[lb]{\smash{{\SetFigFont{12}{14.4}{\familydefault}{\mddefault}{\updefault}{\color[rgb]{0,0,0}$(9|2)$}%
}}}}
\put(1501,-7486){\makebox(0,0)[lb]{\smash{{\SetFigFont{12}{14.4}{\familydefault}{\mddefault}{\updefault}{\color[rgb]{0,0,0}$(17|2)$}%
}}}}
\put(1801,-6136){\makebox(0,0)[lb]{\smash{{\SetFigFont{12}{14.4}{\familydefault}{\mddefault}{\updefault}{\color[rgb]{0,0,0}$(8|3)$}%
}}}}
\put(4501,-3286){\makebox(0,0)[lb]{\smash{{\SetFigFont{12}{14.4}{\familydefault}{\mddefault}{\updefault}{\color[rgb]{0,0,0}$(14|3)$}%
}}}}
\put(4501,-4636){\makebox(0,0)[lb]{\smash{{\SetFigFont{12}{14.4}{\familydefault}{\mddefault}{\updefault}{\color[rgb]{0,0,0}$(15|3)$}%
}}}}
\put(8926,-6136){\makebox(0,0)[lb]{\smash{{\SetFigFont{12}{14.4}{\familydefault}{\mddefault}{\updefault}{\color[rgb]{0,0,0}$(1|2)$}%
}}}}
\put(7351,-6136){\makebox(0,0)[lb]{\smash{{\SetFigFont{12}{14.4}{\familydefault}{\mddefault}{\updefault}{\color[rgb]{0,0,0}$(2|3)$}%
}}}}
\put(6076,-6661){\makebox(0,0)[lb]{\smash{{\SetFigFont{12}{14.4}{\familydefault}{\mddefault}{\updefault}{\color[rgb]{0,0,0}$(3|2)$}%
}}}}
\put(6076,-7036){\makebox(0,0)[lb]{\smash{{\SetFigFont{12}{14.4}{\familydefault}{\mddefault}{\updefault}{\color[rgb]{0,0,0}$(3|3)$}%
}}}}
\put(5026,-6136){\makebox(0,0)[lb]{\smash{{\SetFigFont{12}{14.4}{\familydefault}{\mddefault}{\updefault}{\color[rgb]{0,0,0}$(5|1)$}%
}}}}
\put(3376,-6511){\makebox(0,0)[lb]{\smash{{\SetFigFont{12}{14.4}{\familydefault}{\mddefault}{\updefault}{\color[rgb]{0,0,0}$(7|3)$}%
}}}}
\put(3376,-4786){\makebox(0,0)[lb]{\smash{{\SetFigFont{12}{14.4}{\familydefault}{\mddefault}{\updefault}{\color[rgb]{0,0,0}$(16|3)$}%
}}}}
\put(3376,-2911){\makebox(0,0)[lb]{\smash{{\SetFigFont{12}{14.4}{\familydefault}{\mddefault}{\updefault}{\color[rgb]{0,0,0}$(19|1)$}%
}}}}
\put(3901,-1861){\makebox(0,0)[lb]{\smash{{\SetFigFont{12}{14.4}{\familydefault}{\mddefault}{\updefault}{\color[rgb]{0,0,0}$(11|2)$}%
}}}}
\put(3901,-2386){\makebox(0,0)[lb]{\smash{{\SetFigFont{12}{14.4}{\familydefault}{\mddefault}{\updefault}{\color[rgb]{0,0,0}$(11|3)$}%
}}}}
\put(5851,-1861){\makebox(0,0)[lb]{\smash{{\SetFigFont{12}{14.4}{\familydefault}{\mddefault}{\updefault}{\color[rgb]{0,0,0}$(10|1)$}%
}}}}
\put(1576,-2911){\makebox(0,0)[lb]{\smash{{\SetFigFont{12}{14.4}{\familydefault}{\mddefault}{\updefault}{\color[rgb]{0,0,0}$(12|2)$}%
}}}}
\put(1576,-2386){\makebox(0,0)[lb]{\smash{{\SetFigFont{12}{14.4}{\familydefault}{\mddefault}{\updefault}{\color[rgb]{0,0,0}$(12|1)$}%
}}}}
\put(2776,-6961){\makebox(0,0)[lb]{\smash{{\SetFigFont{12}{14.4}{\familydefault}{\mddefault}{\updefault}{\color[rgb]{0,0,0}$(18;\l|1)$}%
}}}}
\put(3376,-7561){\makebox(0,0)[lb]{\smash{{\SetFigFont{12}{14.4}{\familydefault}{\mddefault}{\updefault}{\color[rgb]{0,0,0}$(18;\l|2)$}%
}}}}
\put(2776,-6136){\makebox(0,0)[lb]{\smash{{\SetFigFont{12}{14.4}{\familydefault}{\mddefault}{\updefault}{\color[rgb]{0,0,0}$(7|2)$}%
}}}}
\put(2401,-5761){\makebox(0,0)[lb]{\smash{{\SetFigFont{12}{14.4}{\familydefault}{\mddefault}{\updefault}{\color[rgb]{0,0,0}$(6|2)$}%
}}}}
\put(2776,-4261){\makebox(0,0)[lb]{\smash{{\SetFigFont{12}{14.4}{\familydefault}{\mddefault}{\updefault}{\color[rgb]{0,0,0}$(16|1)$}%
}}}}
\end{picture}%